\documentclass[reqno]{amsart}
\newcommand \datum {July 13, 2023}
\usepackage{amssymb,latexsym}
\usepackage{amsmath}
\usepackage{graphicx}
\usepackage[dvipsnames]{xcolor}
\usepackage{hyperref}
\usepackage{enumerate}
\numberwithin{equation}{section}
\theoremstyle{plain}
 \newtheorem{theorem}{Theorem}[section]
  
 \newtheorem{lemma}[theorem]{Lemma}
 \newtheorem{proposition}[theorem]{Proposition}

 \newtheorem{corollary}[theorem]{Corollary}
\theoremstyle{definition}
 \newtheorem{definition}[theorem]{Definition}

 \newtheorem{remark}[theorem]{Remark}
 
 \newtheorem{convention}[theorem]{Convention}
\theoremstyle{remark}
 \newtheorem{category}{Category}

                            {\end{enumerate}}   

\newcommand \OT [1] {\textup{OT}(#1)}
\newcommand \EOT [1] {\textup{EOT}(#1)}
\newcommand \LEOT [1] {\textup{LEOT}(#1)}
\newcommand \REOT [1] {\textup{REOT}(#1)}
\newcommand \Str [1] {\textup{Str}(#1)}
\newcommand \TopInt [1] {\textup{TopInt}(#1)}
\newcommand \lsupp [1] {\textup{lsupp}(#1)}
\newcommand \rsupp [1] {\textup{rsupp}(#1)}
\newcommand \LBnd [1] {\textup{LBnd}(#1)}
\newcommand \RBnd [1] {\textup{RBnd}(#1)}
\newcommand \Bnd [1] {\textup{Bnd}(#1)}
\newcommand \Shk [1]  {\mathsf S_7^{(#1)}}
\newcommand \lcorner [1] {\textup{lc}(#1)}
\newcommand \rcorner [1] {\textup{rc}(#1)}
\newcommand \len [1] {\textup{len}(#1)}
\newcommand \con  {\textup{con}}
\newcommand \intv [1]{\mathfrak{#1}}
\newcommand \nn{\intv n}
\newcommand \pp{\intv p}
\newcommand \qq{\intv q}
\newcommand \rr{\intv r}
\newcommand \hh{\intv h}
\newcommand \NTube [1] {\textup{NumTube}(#1)}
\newcommand \pNTube [2] {\textup{NumTube}_{#1}(#2)}
\newcommand \ANTube [1] {\textup{NumTube}_{\textup{all}}(#1)}
\newcommand \INTube [1] {\textup{NumTube}_{\textup{internal}}(#1)}
\newcommand \Floor [1] {\textup{Floor}(#1)}
\newcommand \LFloor [1] {\textup{LFloor}(#1)}
\newcommand \RFloor [1] {\textup{RFloor}(#1)}
\newcommand \Roof [1] {\textup{Roof}(#1)}
\newcommand \LRoof [1] {\textup{LRoof}(#1)}
\newcommand \RRoof [1] {\textup{RRoof}(#1)}
\newcommand \FullRect [1] {\textup{FullRect}(#1)}
\renewcommand \phi{\varphi}
\newcommand \RR {\mathbb R}
\newcommand \Nwl[1] {\textup{Nwl}(#1)}
\newcommand \Nel[1] {\textup{Nel}(#1)}
\newcommand \Min[1] {\textup{Min}(#1)}
\newcommand \Foot [1] {\textup{Foot}(#1)}
\newcommand \Peak [1] {\textup{Peak}(#1)}
\newcommand \CircR [1] {\textup{CircR}(#1)}
\newcommand \UHCircR [1] {\textup{UHCircR}(#1)}
\newcommand \Jir [1] {\textup{J}(#1)} 
\newcommand \Mir [1] {\textup{M}(#1)}
\newcommand \Nplu {\mathbb N^+}
\newcommand \gideal {\mathord{\downarrow_{\textup{g}}}}
\newcommand \gfilter {\mathord{\uparrow_{\textup{g}}}}
\newcommand \ideal {\mathord\downarrow\kern0.5pt }
\newcommand \pideal[1] {\mathord\downarrow_{\kern-1.0pt #1}\kern0.5pt}
\newcommand \pfilter[1] {\mathord\uparrow_{\kern-1.0pt #1}\kern0.5pt}
\newcommand \filter[1]{\mathord\uparrow\kern0.5pt  #1}
\newcommand \GInt [1] {\textup{GInt}(#1)}
\newcommand \then {\Rightarrow}

\newcommand \tbdia {$\mathcal C_1$}

\newcommand \tuple [1] {(#1)}
\newcommand \pair [2] {\tuple{#1,#2}}
\newcommand \Enl [1] {\textup{Lit}(#1)}

\newcommand \Lamp[1] {\textup{Lamp}(#1)}

\newcommand \rhcircr {\boldsymbol\rho_{\textup{CircR}}}
\newcommand \rhfoot {\boldsymbol\rho_{\textup{foot}}}
\newcommand \rhotfoot {\boldsymbol\rho_{\textup{OTfoot}}}
\newcommand \rhotcr {\boldsymbol\rho_{\textup{OTCR}}}
\newcommand \Body [1] {\textup{Body}(#1)}

\DeclareMathOperator{\Con}{Con}
\newcommand{\tbf}{\textbf}
\newcommand{\set}[1]{\{#1\}}

\newcommand \red[1]{{\textcolor{red}{#1}\color{black}}}

\newcommand \semmi[1]{}

\begin{document}

\title[Reducing the lengths of slim semimodular lattices]
{Reducing the lengths of slim planar semimodular lattices without changing their congruence lattices}

\author[G.\ Cz\'edli]{G\'abor Cz\'edli}
\email{czedli@math.u-szeged.hu}
\urladdr{http://www.math.u-szeged.hu/~czedli/}
\address{University of Szeged, Bolyai Institute. 
Szeged, Aradi v\'ertan\'uk tere 1, HUNGARY 6720}

\begin{abstract} 
Following  G.\ Gr\"atzer and E.\ Knapp (2007), a \emph{slim semimodular lattice}, \emph{SPS lattice} for short, 
is a finite planar semimodular lattice having no  $M_3$ as a sublattice. An SPS lattice is a \emph{slim rectangular lattice} if it has exactly two doubly irreducible elements and these two elements are complements of each other. A finite poset $P$ is said to be \emph{JConSPS-representable} if there is an SPS lattice $L$  such that $P$ is isomorphic to the poset $\Jir{\Con L}$ of join-irreducible congruences of $L$. 
We prove that if $1<n\in\mathbb N$ and  $P$ is an $n$-element JConSPS-representable poset, then there exists a slim rectangular lattice $L$ such that $\Jir{\Con L}\cong P$, the length of $L$ is at most $2n^2$, and $|L|\leq  4n^4$. This offers an algorithm to decide whether a finite poset $P$ is JConSPS-representable (or a finite distributive lattice is ``ConSPS-representable"). This algorithm is slow as  G.\ Cz\'edli, T.\ D\'ek\'any, G.\ Gyenizse, and J.\ Kulin proved in 2016 that there are asymptotically $(k-2)!\cdot e^2/2$ many slim rectangular lattices of a given length $k$, 
where $e$ is the famous constant $\approx 2.71828$.
The known properties and constructions of JConSPS-representable posets can accelerate the algorithm; we present a new construction.
\end{abstract}

\thanks{This research was supported by the National Research, Development and Innovation Fund of Hungary under funding scheme K 138892.
}

\subjclass {06C10\hfill{\red{\tbf{\datum}}}}

\dedicatory{Dedicated to the memory of my paternal grandfather,  J\'ozsef}


\maketitle

\section{Introduction} Following Gr\"atzer and E.\ Knapp \cite{GKn-I}, a \emph{slim planar semimodular lattice}, \emph{SPS lattice} for short, is a finite planar (upper) semimodular lattice having no  $M_3$ as a sublattice. By  Gr\"atzer and E.\ Knapp \cite{GKnapp-III}, an  SPS lattice $L$ is a \emph{slim rectangular lattice} if it has exactly two doubly irreducible elements (denoted by $\lcorner L$ and $\rcorner L$ and called the \emph{left corner } and the \emph{right corner} of $L$) and these two elements are complements of each other. As usual, $\Jir L$, the set of join-irreducible elements is $\{x\in L:x$ has exactly one lower cover$\}$; $\Mir L$ is defined dually. As in Cz\'edli and Schmidt \cite{CzGSch-jordanh}, a lattice $L$ is \emph{slim} if it is finite and $\Jir L$ is the union of two chains.
We know from  Cz\'edli and Schmidt \cite[Lemma 2.3]{CzGSch-jordanh} that for a lattice $L$,
\begin{equation}
\text{$L$ is an SPS lattice $\iff$ $L$ is a slim semimodular lattice.}
\label{eq:sZbjmRgRpbW}
\end{equation}
In the paper as in many earlier ones, ``slim semimodular" \emph{means the same as} ``slim planar semimodular'', that is, ``SPS''.
A finite lattice $D$ is \emph{ConSPS-representable}  if it is  isomorphic to the congruence lattice $\Con L$ of an SPS lattice $L$. Similarly, a finite poset $P$ is \emph{JConSPS-representable} if $P\cong \Jir{\Con L}$ for an SPS lattice $L$.

Due to (the historical) Section 2 in  Cz\'edli and Kurusa \cite{CzGKA19}, the surveying part of this section is reduced to a few comments. The  four dozen element list\footnote{See \ \url{ http://www.math.u-szeged.hu/~czedli/m/listak/publ-psml.pdf} \ for an update.} in the  Appendix of Cz\'edli \texttt{https://arxiv.org/abs/2107.10202} shows that since 2007, SPS lattices form an intensively investigated class of lattices. 
In addition to their impact on and connection with geometry, group theory, and combinatorics as explained in \cite{CzGKA19}, 
SPS lattices have connections with finite model theory, see  Cz\'edli \cite{CzGnonfnx}. 
SPS lattices (or their duals) are particular cases of some other classes of lattices and combinatorial structures; indeed, they are also
join-distributive lattices, meet-semidistributive lattices, and 
subspace lattices of antimatroids (or convex geometries); see, for example, Cz\'edli \cite{CzG:Coord}.  Thus, benefiting from the fact that SPS lattices are well understood by means of several structure theorems and representation theorems, the study of these lattices can lead to discoveries for larger classes of lattices and related structures; for example, see Adaricheva and Cz\'edli \cite{KiraCzG}. Actually, even purely geometric papers are in connection with SPS lattices; see, for example, Cz\'edli and Kurusa \cite{CzGKA19}. By Gr\"atzer and Knapp \cite[Section 3]{GKn-I}, the theory of planar semimodular lattices is satisfactorily reduced to that of SPS lattices. So last (and least) we note that 
there are some problems where it could be possible or it was possible to prove more for planar semimodular or SPS lattice than for all finite lattices; see, e.g., Ahmed and Horv\'ath \cite{DelbKHE} and  Cz\'edli and Schmidt \cite{CzGSchTfrankl}.

Within lattice theory, the interest in SPS lattices is mainly fueled by  Gr\"atzer \cite[Problem 1]{gG-cong-fork-ext} asking for a characterization of ConSPS-representable distributive lattices. Note that  \cite[Problem 1]{gG-cong-fork-ext} is motivated by 
the fact that $M_3$ sublattices played a key role in Gr\"atzer, Lakser, and Schmidt \cite{GrLSchm} representing \emph{all} finite distributive lattices by congruence lattices of planar semimodular lattices, whereby it was natural to ask what happens when $M_3$ sublattices are not permitted, that is, when SPS rather than planar semimodular lattices are used.   

Since ConSPS-representability implies distributivity and a finite distributive lattice $D$ is perfectly described by $\Jir D$,
a satisfactory characterization of JConSPS-representable posets would yield a characterization of ConSPS-representable lattices. However, the two representability problems are not the same in the aspect of axiomatizability. Indeed, Cz\'edli \cite{CzGnonfnx}  proves that JConSPS-representable posets cannot be described by finitely many axioms in the first-order language of \emph{finite} posets but it is still unknown whether ConSPS-representable lattices have a finite first-order axiomatization in the class of \emph{finite} lattices.  Note that the class of JConSPS representable posets has many known properties and is closed under some constructions; see Remark \ref{rem:rjmdnMgzKrrGv} for bibliographic details. However, we do not know whether these properties and constructions themselves 
 offer an algorithm to decide whether a poset is JConSPS-representable or not.  Indeed, since we do not know whether  the collection of the above-mentioned known properties and constructions is sound and even a very large SPS lattice can  JConSPS-represent a small poset\footnote{E.g., 
with $\Shk 1, \Shk 2,\dots$ in Figure \ref{fig-enpn}, we have that  $|\Jir{\Con{\Shk k}}|=5$ for all (large) $k$.} $P$, it is not clear at first sight whether it suffices to check $\Jir{\Con L}$ for finitely many $L$.

\section{Goal}
In Theorem \ref{thm:est}, we give an upper bound on the length of the shortest slim rectangular lattices $L$ that JConSPS-represents a  given JConSPS-representable finite poset $P$. 
Therefore, there exists an algorithm to decide if a finite poset $P$ is  JConSPS-representable; indeed, we know from  Cz\'edli, D\'ek\'any, Gyenizse, and Kulin \cite{CzDGyK} that up to isomorphism, 
\begin{equation}\left.
\parbox{10.7cm}{the number of slim rectangular lattices of a given length $k$ is asymptotically $(k-2)!\cdot e^2/2$, where $e=\lim_{n\to\infty}(1+1/n)^n\approx  2.71828$.}\,\,\right\}
\label{eq:gNsztdLnkGsllSskl}
\end{equation}
By \eqref{eq:gNsztdLnkGsllSskl}, there are only finitely many slim rectangular lattices up to a given length. Thus, Theorem \ref{thm:est} implies the existence of an algorithm that for each finite poset $P$ decides whether
$P$ is JConSPS-representable. Moreover, if $P$ is such and $|P|>1$, then the algorithm constructs a slim  rectangular  lattice $L$  such that  $P\cong\Jir{\Con L}$.
Remark \ref{rem:rjmdnMgzKrrGv} points out that known properties and constructions, including the multifork extension construction, make the algorithm faster. Proposition \ref{prop-dbl}
presents a new construction that extends a JConSPS-representable poset to a larger one.

\section{Concepts, terminology, and tools from earlier papers}

As in  Cz\'edli \cite{CzGlamps} and thereafter, to avoid
subscripts of subscripts, the bottom $0_I$ and the top $1_I$ of an interval $I$ are denoted by  $\Foot I$ and $\Peak I$, respectively. 
For $u$ in a lattice $L$, $\ideal u=\pideal L u:=\set{x\in L:x\leq u}$ and $\filter u=\pfilter L u:=\set{x\in L:x\geq u}$.
Edges in a planar diagram are \emph{straight} line segments denoting prime intervals $\pp=[\Foot\pp,\Peak \pp]$. A usual coordinate system of the plane is always fixed.
Edges (or lines) parallel to $(1,1)$ or $(1,-1)$ are of \emph{normal slopes}. Edges  parallel to $(1,t)$ for some $t\in\RR$ with $|t|>1$ and vertical edges are said to be \emph{precipitous}. 

Going after Gr\"atzer and Knapp \cite{GKn-I} and \cite{GKnapp-III},  let $L^\sharp$ be a planar diagram of a slim rectangular lattice $L$. The \emph{left boundary chain} and the \emph{right boundary chain} of $L^\sharp$   are denoted by $\LBnd L$ and $\RBnd L$, respectively. (Actually, $\LBnd{L^\sharp}$ and $\RBnd{L^\sharp}$ would be more precise but we always fix $L^\sharp$ in a way to be defined soon.
This comment applies for several other concepts we are going to define.) The \emph{boundary} of $L$ is  $\Bnd L=\LBnd L\cup\RBnd L$. The elements of $\Bnd L$ and those of $L\setminus\Bnd L$ are called \emph{boundary elements} and \emph{internal elements}. For example, the already mentioned corners are boundary elements: $\lcorner L\in\LBnd L$ and $\rcorner L\in\RBnd L$. For $x\in L$, the \emph{left support} and the \emph{right support} of $x$ are\footnote{The third equality in \eqref{eq:mndLlspSgrnKpL} follows from \eqref{eq:sZbjmRgRpbW} and
 Gr\"atzer and Knapp \cite[Lemmas 3 and 4]{GKnapp-III}.}  
\begin{equation}\left.
\parbox{9cm}{ 
$\lsupp x:=x\wedge \lcorner L$ and $\rsupp x:=x\wedge \rcorner L$. Note that $x=\lsupp x\vee\rsupp x$,  $\lsupp x$ is on the \emph{lower left boundary} $\pideal L{\lcorner L}$, $\pideal L{\lcorner L}\subseteq \LBnd L$, 
$\rsupp x$ is on the \emph{lower right boundary} $\pideal L{\lcorner L}$, and $\pideal L{\rcorner L}\subseteq \RBnd L$.
}\,\,\right\}
\label{eq:mndLlspSgrnKpL} 
\end{equation}
The \emph{upper left boundary} and the \emph{upper right boundary} of $L$ are the principal filters $\pfilter L{\lcorner L}$ and $\pfilter L{\rcorner L}$; note that $\pfilter L{\lcorner L}\subseteq \LBnd L$ and $\pfilter L{\rcorner L}\subseteq \RBnd L$. 

Recall from Cz\'edli \cite[Definition 2.1]{CzGlamps} (as Cz\'edli \cite{CzGdiagrectext} would be too general here) that the diagram $L^\sharp$ of $L$ is a \emph{\tbdia-diagram} if for every edge $\pp=[\Foot\pp, \Peak\pp]$ of the diagram, $\pp$ is either precipitous or it is of a normal slope and, furthermore, $\pp$ is precipitous $\iff$ $\Foot\pp$ is an internal meet-irreducible element of $L$.

\begin{convention}\label{conv:tbdD} 
Together with each slim rectangular lattice occurring in the paper, a
\emph{\tbdia-diagram of our lattice is fixed}. Moreover, even if we do not say it all the time, whenever we construct a lattice (like a sublattice or a larger lattice), then \emph{we always construct its fixed \tbdia-diagram as well}. In notation, we \emph{rarely distinguish} a slim rectangular lattice from its \tbdia-diagram.
\end{convention}

Complying with Convention \ref{conv:tbdD}, all lattice diagrams in this paper are \tbdia-diagrams. 
Let $L$ denote a slim rectangular lattice. Note in advance that quite often,
\begin{equation}
\text{we do not distinguish between lattice theoretic and geometric objects.}
\label{eq:mNdkRkhmzsphvlL}
\end{equation}
If $a<b\in L$ and $C_1,C_2$ are maximal chains of the interval $[a,b]$ such that $C_1\cap C_2=\set{a,b}$ and all elements $x$ of $C_1$ are on the left of $C_2$ (including the possibility of $x\in C_2$), then the elements $[a,b]$ that are  simultaneously on the right of $C_1$ and on the left of $C_2$ form
a so-called \emph{lattice region}; see Kelly and Rival \cite{KR75}
for a more exact definition.
The corresponding geometric area, which is bordered by $C_1$ and $C_2$, is a \emph{geometric region}. Note that whenever we define a geometric area (like a geometric region) or a line segment, then (unless otherwise explicitly stated) it contains its boundaries, that is, it  is \emph{topologically closed}.
Minimal non-chain regions are \emph{cells}. If a cell contains exactly four lattice elements, then it is a \emph{$4$-cell}. Note that 4-cells are cover-preserving boolean sublattices with 4 elements but, as $M_3$ exemplifies, not conversely. 
A \emph{$4$-cell lattice} is a planar lattice in which all cells are 4-cells (in a fixed planar diagram). 
Gr\"atzer and Knapp {\cite[Lemmas 4 and 5]{GKn-I} and \cite{GKnapp-III}} proved that for a planar lattice $L$ (which is finite by definition),
\begin{equation}\left.
\parbox{11cm}{if  $L$ is a $4$-cell lattice, no two distinct $4$-cells have the same bottom,  $L$ has exactly two doubly irreducible elements, and these two elements are complementary, then $L$ is a slim rectangular lattice. Conversely, every slim rectangular lattice is a 4-cell lattice with these properties.}\,\,\right\}
\label{eq:mnHhwlCh}
\end{equation}

\begin{figure}[ht] \centerline{ \includegraphics[scale=0.98]{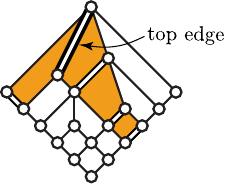}} \caption{A trajectory}\label{fig-loc}
\end{figure}

On the set of prime intervals (i.e., edges) of a slim rectangular lattice $L$, let $\tau$ be the smallest equivalence relation that collapses the opposite sides of every 4-cell. 
As in Cz\'edli and Schmidt \cite{CzGSch-jordanh}, the blocks of $\tau$ are called \emph{trajectories}; e.g., the double-lined edges form a trajectory in Figure \ref{fig-loc}.
Going from left to right, a trajectory does not branch out and neither it does so backwards. 
The unique edge $\pp$ of a trajectory such that $\Foot\pp\in\Mir L$ is the \emph{top edge} of the trajectory.
The \emph{ascending part} of a trajectory consists of the top edge and all of its edges left to the top edge; the \emph{descending part} is defined left-right symmetrically. Any two consecutive edges of a trajectory form a  \emph{$4$-cell of a the trajectory}; they are orange-filled in the figure.

Given a 4-cell $H$ of $L$ and a positive integer $k\in\Nplu$, we obtain the $k$-fold \emph{multifork extension} of $L$ at $H$ by changing $H$ to a copy of $\Shk k$ and proceeding to the southeast and to the southwest to preserve semimodularity. For the exact definition, see Cz\'edli \cite{CzG:pExttCol}, where this construction was introduced, or see Figure \ref{fig-enpn}, where the construction is illustrated by performing a 1-fold multifork extension at $H_1$ of $L_0$ to obtain $L_1$ and performing a 3-fold multifork extension at $H_2$ of $L_1$ to obtain $L_2$. (To save space, our figures are multi-purpose figures;  some ingredients of Figure \ref{fig-enpn} will be explained later.) Note in advance that the thick edges of our lattice diagrams will be called neon tubes.
Note also that 1-fold multifork extensions are also called \emph{fork extensions}; see Cz\'edli and Schmidt \cite{czgscht-visual}; in this case the new elements form a so-called \emph{fork} in the new lattice; see \eqref{eq:lcnFrhvLmTrcJflSr} later.

A \emph{grid} is (the fixed \tbdia-diagram of)  the direct product of two non-singleton finite chains. 
A 4-cell $H$ of $L$ is a \emph{distributive $4$-cell} if the principal ideal $\pideal L{\Peak H}$ is a distributive lattice. By  Cz\'edli and Schmidt \cite{czgscht-visual} and the following lemma, 
\begin{equation}
\text{if $H$ is a distributive $4$-cell of $L$, then   $\pideal L{\Peak H}$ is a grid.}
\label{eq:ltKdlvsTsmnxThj}
\end{equation}
The most useful structure theorem of slim rectangular lattices is the following. 

\begin{lemma}[Multifork Sequence Lemma {\cite[Theorem 3.7]{CzG:pExttCol}}]
\label{lemma:nmtfrkstZ}
For each slim rectangular lattice $L$, there exist positive integers $m_1,\dots, m_k$, a sequence  $L_0$, $L_1$, \dots, $L_k$ of slim rectangular lattices, and a  distributive $4$-cell 
$H_i$ of $L_{i-1}$ for $i\in\set{1,\dots,k}$ such that $L_0$ is a grid, $L_k=L$, and $L_i$ is obtained from $L_{i-1}$ by performing an $m_i$-fold multifork extension at $H_i$ for $i\in\set{1,\dots,k}$. Furthermore, any lattice obtained in this way from a grid is a slim rectangular lattice.
\end{lemma}

The system $(L_0, H_1,m_1, L_1,H_2,m_2,\dots, L_{k-1}, H_k,m_k, L=L_k)$ with components as above is the \emph{multifork sequence of $L$}; it is not necessarily unique but we always fix one. (Note, however, that $k$ is unique.)

\begin{definition}[Cz\'edli \cite{CzG-qdiagr}]\label{def:shfJlrGrDkndmn}Let $\nn$ be an edge on the upper boundary of the initial grid $L_0$. The union of the 4-cells of the trajectory containing $\nn$ is the \emph{original territory} of $\nn$; it is denoted by $\OT\nn$. When we obtain $L_i$ from $L_{i-1}$, then we add several new edges and exactly $m_i$ of these new edges have the same peak as $H$. Let $\nn$ be one of these new edges. In $L_i$, the union of the 4-cells of the trajectory containing $\nn$ is a geometric area; we call it the \emph{original territory} $\OT\nn$ of $\nn$ in $L$.  Note that we have defined $\OT\nn$  if and only if $\nn$ is an edge of the upper boundary or $\nn$ is a precipitous edge. 
If $\nn$ is an edge of the upper left boundary chain, then the \emph{essential part of the original territory}, denoted by  $\EOT\nn$, and the \emph{right essential part of the original territory}, denoted by $\REOT\nn$, of $\nn$ are $\OT \nn$ while the \emph{left essential part of the original territory}, denoted by $\LEOT\nn$, of $\nn$ is the empty set. 
Similarly, for $\nn$ on the right upper boundary, $\EOT \nn=\LEOT \nn:=\OT\nn$ and $\REOT\nn=\emptyset$. 
Next, let $\nn$ be a precipitous new edge of $L_i$ and denote by $T$ the trajectory of $L_i$ that contains $\nn$. The union of the 4-cells of $T$ that do not contain $\nn$ as an edge is the \emph{essential part $\EOT\nn$ of the original territory} of $\nn$;  it is a geometric area and the union of two (geometrically) connected subsets that are, in a self-explanatory manner, called the \emph{left essential part $\LEOT\nn$} and the \emph{right essential part  $\REOT\nn$ of the original territory} of $\nn$. 
\semmi{(The concept of $\EOT  \nn$ was introduced in Cz\'edli \cite{CzG-qdiagr}.)}
\end{definition}

For examples of $\OT\nn$, \dots, $\REOT\nn$, see Figures \ref{fig-enpn}, \ref{fig-knpn}, and \ref{fig-nn}. Even though their  definition relies on $L_0$ or $L_i$, we also use these concepts in $L$, where  $\OT\nn$, \dots, $\REOT\nn$ have no connection with the trajectory containing $\nn$ in general; this is exemplified by $\nn_1$ and $\nn_2$ in $L'$ (but not in $L$) of Figure \ref{fig-knpn}. 
\eqref{eq:ltKdlvsTsmnxThj} implies that 
\begin{equation}\left.
\parbox{10.8cm}{if $\OT\nn$ is defined, then it is bordered by edges of $L$ and all of these edges with peaks different from $\Peak\nn$ are of normal slopes. Furthermore, each of $\LEOT\nn$ and $\REOT\nn$ is either the empty set or a rectangle bordered by edges of normal slopes. See also \eqref{eq:szhmnmthntCrDlHr} later.}\,\,\right\}
\label{eq:lkFnrmRgkCt}
\end{equation}

\semmi{Based on Cz\'edli \cite{CzGlamps}, we recall the following definition.}

\begin{definition}[Cz\'edli \cite{CzGlamps}]\label{def:kzQtnBhMsHl}
Let $L$ be a slim rectangular lattice.
\semmi{; Convention \ref{conv:tbdD} applies.} 

(A) The prime intervals $\pp$ of $L$ with $\Foot\pp\in\Mir L$ are called \emph{neon tubes}. If $\Foot\pp\in\Bnd L$, then $\pp$ is a \emph{boundary neon tube} and it is of a normal slope. Otherwise, $\pp$ is an \emph{internal neon tube} and it is precipitous. (Convention \ref{conv:tbdD} applies.)

(B) Boundary lamps are the same as boundary neon tubes. (However, if $I=\pp$ is a boundary lamp, then we sometimes say that $\pp$ is the neon tube of $I$). An interval $I$ is an \emph{internal lamp} if $\Peak I$ is the peak of an internal neon tube and $\Foot I$ is the meet of the feet of all internal neon tubes with peak $\Peak I$. (These neon tubes are called the neon tubes of $I$.)

(C) In our lattice diagrams (which are \tbdia-diagrams), \emph{the neon tubes are exactly the thick edges} and \emph{the feet of the lamps are black-filled}. We know from Cz\'edli \cite[Lemma 3.1]{CzGlamps} that a lamp is uniquely determined by its foot. Thus, for a lamp $I$,  \emph{we  label the black-filled vertex} $\Foot I$ in our figures by $I$ rather than by $\Foot I$. 
\end{definition}

\begin{figure}[ht] \centerline{ \includegraphics[scale=0.98]{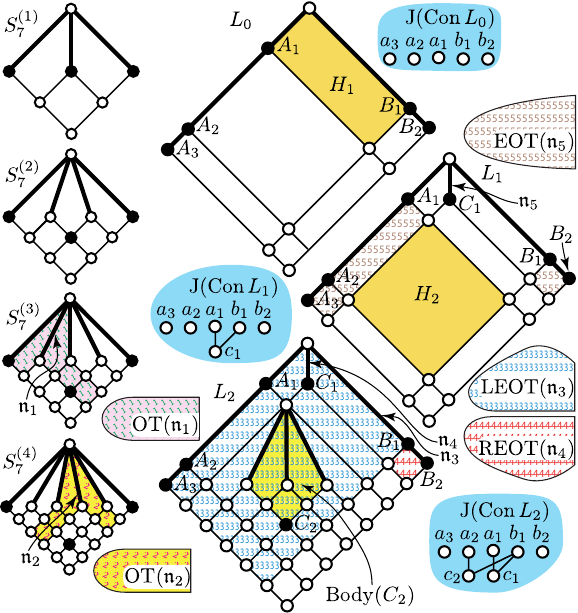}} \caption{Multifork extensions and some geometric objects}\label{fig-enpn}
\end{figure}

Lamps  have been the fundamental tool to study 
JConSPS-representability in  Cz\'edli \cite{CzGlamps}, \cite{CzGnonfnx}, \cite{CzG-DCElamps}, \cite{CzG-qdiagr}, and 
Cz\'edli and Gr\"atzer \cite{CzGGG3p3c}. 
Lamps are particular intervals $I$.  Sometimes,  we need to consider them pairs $(\Foot I,\Peak I)$. 
The (geometric) rectangle  bordered  by $\LBnd L$ and $\RBnd L$ is the \emph{full geometric rectangle} $\FullRect L$ of $L$. 
Combining Definition \ref{def:shfJlrGrDkndmn}
with  Cz\'edli \cite{CzGlamps}, recall the following.

\begin{definition}[Some geometric areas and polygons; Cz\'edli \cite{CzGlamps}]\label{def:mstTjkLvbtn} 
For a slim rectangular lattice (diagram) $L$,  let  $K$ be an interval,  $I$ and $J$ be  lamps, and $\pp$ be a neon tube of $L$. 

(A) The illuminated area $\Enl I$ of $I$ is the union of the original territories of the neon tubes of $I$.

(B) The \emph{left roof} and the \emph{left floor} of the interval $K$ of $L$ are the line segments of slope $(1,1)$ with lower endpoints on the left boundary chain and upper endpoints $\Peak K$ and $\Foot K$, respectively. They are denoted by $\LRoof K$ and $\LFloor K$, respectively. 
With slope $(1,-1)$, the \emph{right roof} $\RRoof K$ and the \emph{right floor} $\RFloor K$ are defined
analogously. The \emph{roof} $\Roof K$ and the \emph{floor} $\Floor K$ of $K$ are $\LRoof K\cup\RRoof K$ and $\LFloor K\cup \RFloor K$, respectively. 

(C) For a set $X$ of planar points,  $\GInt X$ stands for the \emph{geometric} (i.e., topological) \emph{interior} of $X$.  Let $h$ be a (geometric) polygon with endpoints $a$ and $b$ such that  $h\setminus\set{a,b}\subseteq \TopInt{\FullRect L}$,  $a\in\LBnd L$, and $b\in\RBnd L$.  Then $h$ cuts $\FullRect L$ into an upper half $\gfilter h$ and a lower half $\gideal h$; by convention, $h=\gfilter h\cap\gideal h$. Note that $\Enl I=\gfilter{\Floor I}\cap\gideal{\Roof I}$, and similarly for $\Enl{\pp}$. 

(D) The \emph{body} $\Body I$ of $I$ is the geometric region determined by $I$; if $I$ has only one neon tube, then $\Body I$ is a line segment. For example, in Figure \ref{fig-enpn}, $C_2\in\Lamp{L_2}$ and  $\Body{C_2}$ is yellow-filled.

(E) If $I$ is a internal lamp, then the \emph{circumscribed rectangle} $\CircR I$ is
the region determined by the interval $[x,\Peak I]$ where $x$ is the meet of the leftmost lower cover and the rightmost lower cover of $\Peak I$. (Equivalently, $x$ is the meet of all lower covers of $\Peak I$.)
\end{definition}

Since the edges occurring in Definition \ref{def:shfJlrGrDkndmn} are the same as the neon tubes of $L$, the following lemma in the present setting is not surprising.

\begin{lemma}[Cz\'edli {\cite[(2.10)]{CzGlamps}}]\label{lemma:mfXzsrjsmG}
For the fixed multifork sequence of $L$, see Lemma \ref{lemma:nmtfrkstZ}, the set of internal lamps of $L$ is $\set{I_j:1\leq j\leq k}$ where, for $j\in\set{1,\dots,k}$, the lamp $(\Foot{I_j}, \Peak{I_j})$ comes to existence by the $j$-th multifork extension,  $\CircR {I_j}$ in $L=L_k$ is the geometric region determined by $H_j$ in $L_{j-1}$, and $\Foot{I_{j}}\in L_j\setminus L_{j-1}$. 
\end{lemma}

Since the multifork extensions in Lemma \ref{lemma:nmtfrkstZ} are performed at \emph{distributive} 4-cells, it follows easily that, using the notations of Lemma \ref{lemma:mfXzsrjsmG}, for any $j\in\set{1,\dots,k}$, 
\begin{equation}\left.
\parbox{9.5cm}{the lower covers of $\Peak {I_j}$ are the same in $L_j$ as in $L=L_k$. In particular, $I_j$ has the same neon tubes in $L_j$ as in $L$. Furthermore, if a neon tube $\nn$ comes to existence in $L_j$, then $\EOT\nn$, $\LEOT\nn$, and $\REOT\nn$ are the same in $L_j$ as in $L$.}\,\,\right\}
\label{eq:szhmnmthntCrDlHr}
\end{equation}

\begin{definition}\label{def:hzjJLpgGdjR}
With the notation used in Lemma \ref{lemma:mfXzsrjsmG}, let $I_i$ and $I_j$ be lamps of $L$. If $i<j$, then we say that $I_j$ is \emph{younger} than $I_i$ and $I_i$ is \emph{older} than $I_j$. (This concept depends on the multifork sequence, but this sequence is always fixed.)
\end{definition}

By an \emph{edge segment} we mean a geometric line segment $\intv g$ of positive length with endpoints lying on the same edge  $\intv e$ of (the fixed \tbdia-diagram of) $L$. In this case, we say that $\intv g$ is an \emph{edge segment of} $\intv e$. Based on the fact that the neon tubes of $L$ are exactly the prime intervals occurring in Definition \ref{def:shfJlrGrDkndmn}, we can recall a part of  Cz\'edli \cite[Definition 2.9]{CzGlamps}
and extend it as follows.

\begin{definition}\label{def:kzpnjvdlkgrGsMb} Let $I$ and $J$ be lamps of a slim rectangular lattice $L$.

(A) Let $\pair I J\in\rhfoot$ mean that $I\neq J$, $\Foot I\in\Enl J$, and $I$ is an internal lamp.

(B) Let $\pair I J\in\rhotfoot$ mean that $I\neq J$, $I$ is an internal lamp, and $J$ has a neon tube $\nn$ such that $\Foot I\in\GInt{\LEOT\nn}$ or $\Foot I\in\GInt{\REOT\nn}$.

(C) Let $\pair I J\in\rhotcr$ mean that $I\neq J$, $I$ is an internal lamp, and $J$ has a neon tube $\nn$ such that $\CircR I\subseteq \LEOT\nn$ or $\CircR I\subseteq \REOT\nn$.

(D) Let $\pair I J\in\rhcircr$ mean that $I\neq J$, $I$ is an internal lamp, and $\CircR I\subseteq\Enl J$.

(E) Let $\Lamp L$ be the set of lamps of $L$, and let ``$\leq$'' be the reflexive and transitive closure of the relation $\rhfoot$.
The relational structure $(\Lamp L;\leq)$ is also denoted by $\Lamp L$.
\end{definition}

The congruence generated by a pair $(x,y)$ of elements will be denoted by $\con(x,y)$.

\begin{lemma}[Mostly Cz\'edli {\cite[Lemma 2.11]{CzGlamps}}]\label{lemma:vltfLm} If $L$ is a slim rectangular lattice, then $\rhfoot$ $=$  
$\rhcircr$ $=$  $\rhotfoot$ $=$ $\rhotcr$, 
$\Lamp L=(\Lamp L;\leq)$ is a poset, and
whenever  $I\prec J$ in $\Lamp L$, then $(I,J)\in\rhfoot$. Furthermore, we have that $(\Lamp L;\leq) \cong (\Jir{\Con L};\leq)$ and the map
\begin{equation}
\text{$\phi\colon (\Lamp L;\leq) \to (\Jir{\Con L};\leq)$ defined by $I\mapsto \con(\Foot I,\Peak I)$}
\label{eq:sznZtcsTl}
\end{equation} 
is an order isomorphism.
\end{lemma}

The advantage of this lemma over its precursor, \cite[Lemma 2.11]{CzGlamps}, is that   $(I,J)\in\rhfoot$ is a mild condition, which is easy to verify, while $(I,J)\in\rhotcr$ is a strong condition, which gives more chance to draw conclusion from. 

\begin{proof}
With the exception of ``$\rhfoot$ $=$ $\rhotfoot$ $=$ $\rhotcr$'', the lemma is already known; see Cz\'edli \cite[Lemma 2.11]{CzGlamps}. So we need only to show the just-mentioned equalities. 
Clearly, $\rhotcr \subseteq \rhotfoot\subseteq \rhfoot$. Assume that $I_i,I_j\in \Lamp L$ such that $(I_i,I_j)\in\rhfoot$.  Since $\Shk {m_{i}}$ is not distributive, it follows from \eqref{eq:ltKdlvsTsmnxThj} and  Lemmas \ref{lemma:nmtfrkstZ} and \ref{lemma:mfXzsrjsmG} that $I_i$ is younger than  $I_j$, that is, $i>j$. 
In particular, $I_i$ is an internal lamp.
With $m:=m_j$, let $\nn_1,\dots,\nn_m$ be the neon tubes of $I_j$. As $i>j$, these neon tubes are present in $L_{i-1}$, and so are their original territories $\OT{\nn_1}$,\dots, $\OT{\nn_m}$ as well as their essential original territories; see \eqref{eq:szhmnmthntCrDlHr}. By \eqref{eq:lkFnrmRgkCt} applied to $L_{i-1}$, these territories are separated by polygons consisting of lattice edges. By planarity, these ``separating polygons'' cannot cross the 4-cell $H_i$ of $L_{i-1}$; this 4-cell becomes $\CircR{I_i}$ in $L_i$ and in $L$. So $\CircR{I_i}\subseteq \OT{\nn_t}$ for some $t\in\set{1,\dots,m}$. But the 4-cell $H_i$ in question cannot have the same top as $I_j$ since 
the opposite case would contradict the distributivity of $H_i$ in $L_{i-1}$. (Alternatively, \cite[Lemma 6.2]{CzG-qdiagr} would also lead to a contradiction.) Hence, $\CircR{I_i}=H_i\subseteq \EOT{\nn_t}$. Since $\EOT {\nn_t}$ is the union of its two connected ``components", $\LEOT {\nn_t}$ and $\REOT {\nn_t}$, and these  components are in a positive geometric distance from each other (provided none of them is the empty set), the planarity of the diagram yields that $\CircR{I_i}=H_i\subseteq \LEOT{\nn_t}$ or $\CircR{I_i}=H_i\subseteq \REOT{\nn_t}$. Hence, 
$\pair {I_i}{I_j}\in \rhotcr$, implying that $\rhotcr\subseteq \rhfoot$ and completing the proof of Lemma \ref{lemma:vltfLm}.
\end{proof}

Since we work with the \tbdia-diagram of our slim rectangular lattice $L$, the illuminated sets $\Enl I$ and the $\Foot I$, and so the relation $\rhfoot$ are perfectly described by the geometric structure
\begin{align}
&\left.\Str L:= \Bigl(
\FullRect L, \set{(\Foot I,\Peak I): I\in\Lamp L}
\Bigr).\right.
\label{eq:zgtlnRdkrtkZnt}
\\
&\left.
\parbox{10.5cm}{In particular, if  $L$ and $L'$ are slim rectangular lattices such that $\Str L=\Str{L'}$, then $\Lamp L\cong\Lamp{L'}$ and so 
$\Con L\cong\Con{L'}$.}\,\,\right\}
\label{eq:NmdTlskDmbzK}
\end{align}

\section{Auxiliary statements}

\begin{figure}[ht] \centerline{ \includegraphics[width=0.98\textwidth]{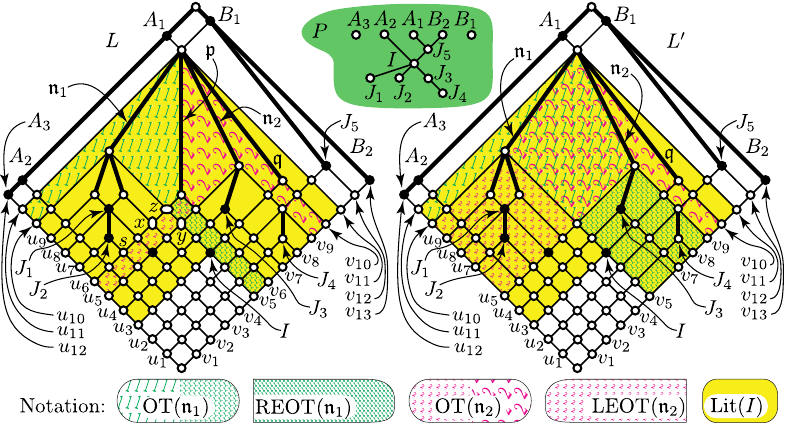}} \caption{Illustrating the proof of Lemma \ref{lemma:middllunt} by $\Lamp L\cong P\cong\Lamp{L'}$}\label{fig-knpn}
\end{figure} 

The following definition is motivated by $\rhotcr$; see  Definition \ref{def:kzpnjvdlkgrGsMb} and Lemma~\ref{lemma:vltfLm}.

\begin{definition}\label{def:mprdClgRtRblc} 
For a  slim rectangular lattice $L$ and  $J\in\Lamp L$, let $\pp$ be a neon tube of $J$. We say that \emph{the original territory of $\pp$ is used} if 
 there is a lamp $I\in \Lamp  L$ such that 
$I\neq J$ and $\CircR I\subseteq \LEOT\pp$ or $\CircR I\subseteq \REOT\pp$.  
If $I$ is such, then we say that $I$ \emph{uses the original territory} of $\pp$. 
If there is no such $I$, then \emph{the original territory  of $\,\pp$ is not used}.
\end{definition}

\begin{remark}\label{rem:nshmlmfPvhMrGnk}
Lemma \ref{lemma:vltfLm} implies that 
in Definition \ref{def:mprdClgRtRblc}, 
``$I\neq J$'' is equivalent to ``$I<J$''. Furthermore,  $I\neq J$  occurs in Definition \ref{def:mprdClgRtRblc} only for emphasis, so it could be omitted; analogous comments would apply to Lemma \ref{lemma:thtznjtVnPrm} below.
\end{remark}

\begin{lemma}\label{lemma:thtznjtVnPrm}
For $\pp$ and $J$ as in Definition \ref{def:mprdClgRtRblc}, the following four conditions are equivalent. 

\textup{(a)} The original territory of 
$\pp$ is used, 
that is, there is lamp $I$ such that $\pp$ is not a neon tube of $I$ 
and $\CircR I\subseteq \LEOT\pp$ or $\CircR I\subseteq \REOT\pp$. 

\textup{(b)} There is a lamp $I\in\Lamp L\setminus\set J$ such that $\Foot I$ is in $\GInt{\LEOT\pp}$ or it is in $\GInt{\REOT\pp}$.

\textup{(c)} There is a lamp $I\in\Lamp L\setminus\set J$ such that $\Foot I$ is in $\EOT\pp$.

\textup{(d)} There is a precipitous edge segment in $\EOT \pp$. 

Furthermore, if a lamp $I$ satisfies 
one of \textup{(a)}, \textup{(b)}, and \textup{(c)}, then it satisfies all the three.
\end{lemma}

\begin{proof} Since we never change $I$ to another lamp, the last sentence of the lemma will automatically follow when
the equivalence of (a), (b), and (c) has been proved.

Since $\Foot I\in\GInt{\CircR I}$, (a) implies (b). By the equality $\EOT\pp=\LEOT\pp\cup\REOT\pp$, we obtain that (b) implies (c). 

Next, assume that (c) holds. Then  $\Foot I\in \EOT\pp \subseteq \Enl J$ and so $(I,J)\in\rhfoot$. By Lemma \ref{lemma:vltfLm},  $(I,J)\in\rhotcr$ and so $\Body I\subseteq \CircR I\subseteq \Enl I$.
Thus, $I_t:=I$ is younger than $I_k:=J$ in the sense of Definition \ref{def:hzjJLpgGdjR}, that is, $t>k$;  indeed, if $I=I_t$ was older than $J=I_k$, then the 4-cell $H_k$ would not be distributive in $L_{k-1}$. In $L_k$,  each of $\LEOT\pp$, $\REOT\pp$, and 
$\FullRect{L_k}\setminus \EOT\pp$ were unions of 4-cells. Some of these 4-cells  could have been divided into smaller ones later, but even in $L_{t-1}$, 
each of $\LEOT\pp$, $\REOT\pp$, and 
$\FullRect{L_{t-1}}\setminus \EOT\pp$
were unions of  4-cells. Hence, 
$H_t\subseteq \LEOT\pp$, $H_t\subseteq \REOT\pp$,  or $H_t$ is outside $\EOT\pp$. Since $\Foot{I}=\Foot{I_t}\in\GInt{H_t}$ and $\Foot I\in\EOT\pp$,  $H_t$ was not outside $\EOT\pp$. Hence, $\CircR I=\CircR{I_t}=H_t\subseteq \LEOT\pp$ or 
 $\CircR I\subseteq \REOT\pp$, whereby the original territory of $\pp$ is used. 
Thus, (c) implies (a), and we have proved that (a), (b), and (c) are equivalent conditions. 

By Remark \ref{rem:nshmlmfPvhMrGnk}, the implication (a) $\then$ (d) is trivial. 

Finally, assume that (d) holds. Then we have a precipitous edge segment in $\LEOT\pp$ or in $\REOT\pp$, say, in $\LEOT\pp$. By the second half of \eqref{eq:lkFnrmRgkCt}, we can assume that a precipitous edge segment lies in $\GInt{\LEOT\pp}$. This edge segment lies on a neon tube $\qq$ of a lamp $I$. By planarity and \eqref{eq:lkFnrmRgkCt}, $\qq$ cannot cross the four sides bordering (the geometric rectangle) $\LEOT\pp$, so $\qq$ lies fully in $\LEOT\pp$. In particular, $\Peak I=\Peak \qq\in\LEOT\pp$ and $\Foot\qq\in\LEOT\pp$. Observe  that $\Peak I$ cannot lie on the lower boundary of $\LEOT\pp$ since otherwise $\qq$, going down from $\Peak I$ with a precipitous slope, could not include an edge segment lying in $\LEOT\pp$. 

Next, let $\rr$ be an arbitrary neon tube of $I$. It goes down from $\Peak\rr=\Peak I$ with a precipitous slope. Thus, since $\Peak\rr$ is not on the lower boundary, \eqref{eq:lkFnrmRgkCt} yields that an edge segment lying on $\rr$ lies also in $\GInt{\LEOT\pp}$. So $\rr$ satisfies the same condition as $\qq$ above, and it follows that $\Foot\rr\in\LEOT\pp$. 

Now let $\rr'$ and $\rr''$ be the leftmost neon tube and the rightmost neon tube of $I$. If $\rr'= \rr''$, then $\qq$ is the only neon tube of $I$, and the required $\Foot I\in \LEOT\pp$ follows from $\Foot I=\Foot\qq\in\LEOT\pp$. So we can assume that $\rr'\neq\rr''$. 
Then $\Foot{\rr'}$ and $\Foot{\rr''}$, as distinct lower covers of $\Peak I$, are incomparable; see \eqref{eq:szhmnmthntCrDlHr}. 
By the main result of Cz\'edli \cite{CzGpropmeet}
and $\Foot I=\Foot{\rr'}\wedge \Foot{\rr''}$, the interval
$[\Foot I,\Foot{\rr'}]$ is a chain (and so a line segment) of slope $(1,-1)$ while 
$[\Foot I,\Foot{\rr''}]$ is a line segment of slope $(1,1)$. 
The top endpoints $\Foot{\rr'}$ and $\Foot{\rr''}$ of these line segments are in $\LEOT\pp$, whereby so is their common bottom $\Foot I$ by the second half of \eqref{eq:lkFnrmRgkCt}. Hence, $\Foot I\in\LEOT\pp$, that is, (a) holds. This completes the proof of Lemma \ref{lemma:thtznjtVnPrm}.
\end{proof}

Let $\pp$ be an internal neon tube of a slim rectangular lattice $L$. As in  Cz\'edli and Schmidt \cite{czgscht-visual} (but with different terminology), the \emph{fork  determined by $\pp$} is
\begin{equation}\left.
\parbox{10cm}{$F(\pp):=[\lsupp{\Foot\pp},\Foot\pp]\cup [\rsupp{\Foot\pp},\Foot\pp]$ together with the edges of these two intervals and the edge $\pp$.}\,\,\right\}
\label{eq:lcnFrhvLmTrcJflSr}
\end{equation}
For the  particular case when  $\pideal{L'}{\Peak \pp}$ is distributive, the following lemma occurs implicitly in  \cite{czgscht-visual}.

\begin{lemma} \label{lemma:mThFgZzbBRt}
If $\pp$ is a neon tube of a slim rectangular lattice $L$ and $L':=L\setminus F(\pp)$, see \eqref{eq:lcnFrhvLmTrcJflSr}, then $L'$ is meet-subsemilattice of $L$.
\end{lemma}

\begin{proof}
First, we prove that 
\begin{equation}
[\lsupp{\Foot\pp},\Foot\pp] =\set{x\in L: \lsupp x=\lsupp{\Foot \pp}}.
\label{eq:znfLdjfrGlTkr}
\end{equation}
Denote $\Foot\pp$ by $w$ and  $\lsupp{\Foot\pp}$ by $u$; so $u=\lsupp w$ and we need to show that $[u,w]=\set{x\in L: \lsupp x=u}$. For $y\in[u,w]$, we have that $u=\lsupp u\leq\lsupp y\leq \lsupp w=u$. Hence, $y\in \set{x\in L: \lsupp x=u}$ and we obtain that 
$[u,w]\subseteq\set{x\in L: \lsupp x=u}$.
To exclude that ``$\subset$'' holds here, suppose for contradiction that there is a $z\in \set{x\in L: \lsupp x=u}$ such that $z\notin [u,w]$.
Then  
$ z=\lsupp z\vee \rsupp z=u\vee\rsupp z$ implies that $u\leq z$, and if $\rsupp z\leq \rsupp w$, then $z\leq u\vee \rsupp w\leq w$ would contradict that $z\notin[u,w]$.
But $\rsupp z$ and $\rsupp w$ belong to the same chain, $\RBnd L$, so they are comparable, and we obtain that $\rsupp w < \rsupp z$. 
Hence, $w=\lsupp w\vee \rsupp w=u\vee\rsupp w\leq u\vee \rsupp z=\lsupp z\vee\rsupp z=z$. Now the inequality $w\leq z$ and $z\notin[u,w]$ imply that $\Foot\pp=w<z$. Taking the meet-irreducibility of $\Foot\pp$ into account, we have that $\Peak \pp\leq z$. Thus, 
$\lsupp{\Peak\pp}\leq \lsupp z$.
With the notation used in Lemmas \ref{lemma:nmtfrkstZ} and \ref{lemma:mfXzsrjsmG}, let $I_i$ be the lamp to which $\pp$ belongs. Then $\Peak\pp=\Peak {I_i}$, and it is clear in $L_i$ that $u=\lsupp{\Foot \pp} < \lsupp{\Peak I}=\lsupp{\Peak\pp}$. Since $L_i$ is a sublattice of $L$, the inequality $u<\lsupp{\Peak \pp}$ also holds in $L$. 
Combining this with the already established $\lsupp{\Peak\pp}\leq \lsupp z$, we obtain that 
$u<\lsupp z$. This contradicts the assumption 
$z\in \set{x\in L: \lsupp x=u}$ and proves \eqref{eq:znfLdjfrGlTkr}. 

Next,  for the sake of contradiction, suppose that $L'$ is not meet-closed. Pick elements $s,c,d\in L$ such that $s=c\wedge d$, 
$s\in F(\pp)=L\setminus L'$ but $c,d\notin F(\pp)$. By \eqref{eq:lcnFrhvLmTrcJflSr}, \eqref{eq:znfLdjfrGlTkr}, and symmetry, we can assume that $\lsupp s=\lsupp{\Foot p}$. Since the function $L\to\LBnd L$ defined by $t\mapsto \lsupp t$ is clearly an idempotent meet-endomorphism by \eqref{eq:mndLlspSgrnKpL},  
$\lsupp s=\lsupp c\wedge \lsupp d$. 
As $\LBnd L$ is a chain,  $\lsupp s \in\set{\lsupp c,\lsupp d}$. Let, say, $\lsupp s=\lsupp c$.
Then $\lsupp c=\lsupp{\Foot p}$, so \eqref{eq:lcnFrhvLmTrcJflSr} and \eqref{eq:znfLdjfrGlTkr} give that $c\in F(\pp)$, a contradiction.
\end{proof}

For $I\in\Lamp L$, let $\NTube  I =  \pNTube L I$ denote the number of neon tubes of $I$. The total number of neon tubes of $L$ is denoted by  $\ANTube L$, so
$\ANTube L:=\sum_{I\in\Lamp L}\NTube I$.

\begin{lemma}[Sandwiched Neon Tube Lemma]\label{lemma:middllunt} For a  slim rectangular lattice $L$, let $\nn_1$, $\pp$, and $\,\nn_2$ be three consecutive neon tubes of an internal lamp $I\in\Lamp L$ 
such that the original territory of $\pp$ is used but those of $\nn_1$ and $\nn_2$ are not used. Then there is a slim rectangular lattice $L'$ such that
$\Lamp {L'}\cong \Lamp L$ but $|L'|<|L|$  and   $\ANTube{L'}=\ANTube L -1$; in fact, there is an isomorphism $\phi\colon\Lamp L\to\Lamp{L'}$ such that $\NTube{\phi(I)}=\NTube I-1$ and $\NTube{\phi(J)}=\NTube J$ for all $J\in\Lamp L\setminus\set I$. 
\end{lemma}

\begin{proof} With reference to \eqref{eq:lcnFrhvLmTrcJflSr}, denote by $L'$ the subposet of $L$ that we obtain from $L$ by removing the fork $F(\pp)$ determined by $\pp$; see Figure \ref{fig-knpn} for an illustration. We are going to show that $L'$ does the job.
By left-right symmetry, we can assume that $\nn_1$ is to the left of $\pp$ and $\pp$ is to the left of $\nn_2$.

First, we prove that $L'$ is a sublattice.
By the main result of Cz\'edli \cite{CzGpropmeet}, 
\begin{equation}\left.
\parbox{8cm}{both intervals occurring in \eqref{eq:lcnFrhvLmTrcJflSr} are chains of normal slopes. Hence, by \eqref{eq:mNdkRkhmzsphvlL},  $F(\pp)=\Floor\pp$.}\,\,\right\}
\label{eq:szmRkhs}
\end{equation}
In Figure \ref{fig-knpn}, these chains are $[u_6, u_6\vee v_6]$ and   $[v_6, u_6\vee v_6]$. 
Since none of the original territories of $\nn_1$ and $\nn_2$ are used, we obtain from Lemma \ref{lemma:thtznjtVnPrm} that
\begin{equation}
\text{none of $\REOT{\nn_1}$ and  $\LEOT{\nn_2}$ contains a precipitous line segment.}
\label{eq:wGzkBkLkK}
\end{equation}
These two areas border $F(\pp)=\Floor\pp$ from below. Thus, for any edge $\rr$ of $L$,
\begin{equation}
\text{if $\Peak\rr\in F(\pp)$, then $\rr$ is of a normal slope.}
\label{eq:nsmmRtvnbHrjSt}
\end{equation}
For the sake of contradiction, suppose that $L'$ is not join-closed. Then we can pick $x',y'\in L'$ such that $z:=x'\vee y'\notin L'$, that is, $z\in F(\pp)$. (The join is taken in $L$.) By \eqref{eq:lcnFrhvLmTrcJflSr} and left-right symmetry, we can assume that $z\in[\lsupp {\Foot \pp}, \Foot \pp]$. 
In Figure \ref{fig-knpn}, the situation is illustrated with $z$ as the (unique) element drawn by a lying oval. 
Let $T:=[\lsupp {\Foot{\nn_2}},  \Foot{\pp}]$ in $L$; it is $[u_5, u_6\vee v_6]$ in Figure \ref{fig-knpn}. The (area determined by) $T$ is $\LEOT{\nn_2}\subseteq \EOT{\nn_2}$. Hence, by \eqref{eq:wGzkBkLkK}, $T$ contains no precipitous line segment. Furthermore,
as a lattice interval, 
\begin{equation}
\text{$T$ is the direct product of a chain and the two-element chain.}
\label{eq:mhgPrcmgrndkKns}
\end{equation}
Hence, $z$ has only two lower covers, $x$ and $y$ (the standing ovals in the figure), and the edges $[x,z]$ and $[y,z]$ are of normal slopes. Let, say, $x$ be to the left of $y$. Now $x',y'\in \pideal L z\setminus\set z$, but $\set{x',y'}\nsubseteq \pideal L {y}$ since otherwise $z=x'\vee y'\leq y\prec z$ would be a contradiction. Hence, at least one of $x'$ and $y'$ is in $\pideal L z\setminus \pideal L y\subseteq [\lsupp {\Foot \pp}, \Foot \pp]\subseteq F(\pp)$, contradicting that $x',y'\in L'=L\setminus F(\pp)$. Therefore, $L'$ is closed with respect to joins. Since it is also closed with respect to meets by Lemma \ref{lemma:mThFgZzbBRt}, we have proved that 
 $L'$ is a sublattice of $L$. 

Let $\intv e$ be an edge in the interval $[\lsupp{\Foot\pp},\Foot\pp]$ distinct from the top edge of this interval. Using \eqref{eq:mhgPrcmgrndkKns}, it is clear that if we merge the two 4-cells that share $\intv e$ as a common side, we obtain a 4-cell of $L'$. The situation is similarly for the non-top edges of $[\rsupp{\Foot\pp},\Foot\pp]$. The top edges of these two intervals disappear when $\Foot\pp$ and its two lower covers are omitted and three ``old'' 4-cells merge into a ``new'' 4-cell of $L'$. Now that we have described the new 4-cells, it follows from \eqref{eq:mnHhwlCh} that $L'$ is a slim rectangular lattice. 

It is clear by the paragraph above that with the exception of $\pp$, only some edges of normal slopes are removed when passing from $L$ to $L'$. The removal of $\pp$ does not influence the pair $(\Foot I, \Peak I)$ since $\Foot I$ is the meet of the feet of the leftmost neon tube and the rightmost neon tube of $I$ but $\pp$ is a ``middle'' neon tube of $I$. Therefore, $\Str{L'}=\Str L$, see \eqref{eq:zgtlnRdkrtkZnt}, and so \eqref{eq:NmdTlskDmbzK} implies that $\Lamp{L'}\cong \Lamp L$. Finally, since only one neon tube, $\pp$, has been removed,  $\ANTube{L'}=\ANTube L -1$. The existence of $\phi$ is clear: for $J\in\Lamp L$,  $\phi(J)$ is defined by the property $(\Foot{\phi(J)},\Peak{\phi(J)})=(\Foot J, \Peak J)$.
The proof of Lemma \ref{lemma:middllunt} is complete.
\end{proof}

\begin{figure}[ht] \centerline{ \includegraphics[width=0.98\textwidth]{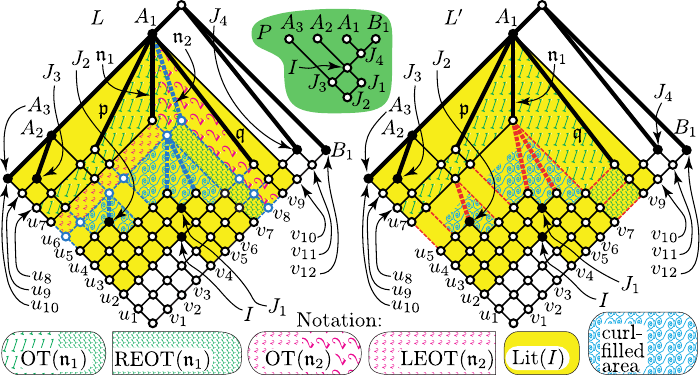}} \caption{Illustrating the proof of Lemma \ref{lemma:twnTbs} by $\Lamp L\cong P\cong\Lamp{L'}$}\label{fig-nn}
\end{figure}

\begin{lemma}[No Neighboring Neon Tubes Lemma]\label{lemma:twnTbs} 
Let $L$ be a  slim rectangular lattice. Assume that  $\nn_1$ and $\,\nn_2$ are two neighboring neon tubes of an internal lamp $I\in\Lamp L$ such that their original territories are not used.
Then there exists a slim rectangular lattice $L'$ such that  $|L'|<|L|$ and 
$(\Lamp{L'};\leq)\cong (\Lamp{L};\leq)$
but $|\ANTube{L'}| = |\ANTube{L}|-1$; in fact, 
there is an order isomorphism $\phi\colon (\Lamp L;\leq) \to (\Lamp{L'};\leq)$ such that $|\NTube {\phi(I)}|=
|\NTube I|-1$ but $|\NTube {\phi(K)}|= |\NTube K|$  for any $K\in\Lamp L\setminus\set{I}$.
\end{lemma}

\begin{proof} 
The proof borrows some ideas from Cz\'edli \cite{CzG-qdiagr}. Note, however, that the present situation  is different from that in  \cite{CzG-qdiagr} since now $L'$, to be defined below, is not a quotient lattice of $L$ in general.

Let, say, $\nn_2$ be to the right of $\nn_1$; see Figure \ref{fig-nn} for an illustration. 
Observe that, by Lemma \ref{lemma:thtznjtVnPrm} (or see the figure) and the fact that  $\REOT{\nn_1}$ is not used, 
\begin{equation}\left.
\parbox{10.5cm}{the peak of no precipitous edge of $L$  belongs to $\RFloor {\nn_2}$ and, in particular,  $\Foot{\nn_2}$ cannot be the peak of a precipitous edge of $L$.}\,\,\right\}
\label{eq:XlsnmmrdhlsfTRhs}
\end{equation} 
Keeping Convention \ref{conv:tbdD} in mind, we define $L'$ by describing its \tbdia-diagram. From (the diagram of) $L$, we remove the fork $F(\nn_2)$ together with all edges that have one or two endpoints in $F(\nn_2)$.  Writing this formally, $L'=L\setminus F(\nn_2)$. 
On the left of Figure \ref{fig-nn}, the vertices 
to be omitted are drawn in blue while the edges to be omitted are the blue dashed edges. Let $L'$ be the set of the remaining vertices (drawn in black). (Note that $L'$ in Figure \ref{fig-nn} is not a sublattice of $L$ since $u_4,v_6\in L'$ but $u_4\vee_L v_6\notin L'$.)
At this stage, $L'$ with the remaining (black solid) edges is not even a lattice diagram. 

Next, let $\qq$ denote the right neighbor of $\nn_2$ among the neon tubes of $I$ or, if $\nn_2$ is the rightmost neon tube of $I$, then let $\qq$ be the upper right edge of $\CircR I$. Actually, it is only $\Foot\qq$ that we will need, and it is the right neighbor of $\Foot{\nn_2}$ among the lower covers of $\Peak{\nn_2}=\Peak I$. For each edge $\rr$ of $L$,  we define or not define an edge $\rr'$ of $L'$ as follows.
\allowdisplaybreaks{
\begin{align}
\left.
\parbox{10cm}{If $\Foot\rr\in\Floor{\nn_2}$, then $\rr'$ is undefined and $\rr$ is called an \emph{omitted old edge}.}\,\,\right\} \label{eq:sRgTbrhtza}
\\
\left.
\parbox{10cm}{If  $\Foot\rr\notin\Floor{\nn_2}$ and $\Peak\rr\notin \Floor{\nn_2}$, then $\rr':=\rr$  and $\rr$ is called a \emph{remaining old edge} of $L'$.}\,\,\right\} \label{eq:sRgTbrhtzb}
\\
\left.\parbox{10cm}{If $\Foot\rr\notin\Floor{\nn_2}$ and 
$\Peak\rr\in \LFloor{\nn_2}$, 
then let $\Foot{\rr'}:=\Foot{\rr}$ and $\Peak{\rr'}:=\Peak{\rr}\vee_L \lsupp{\Foot{\nn_1}}$.}\,\,\right\} \label{eq:sRgTbrhtzc}
\\
\left.
\parbox{10cm}{If $\Foot\rr\notin\Floor{\nn_2}$ and $\Peak\rr\in \RFloor{\nn_2}$, then let  $\Foot{\rr'}:=\Foot{\rr}$ and $\Peak{\rr'}:=\Peak{\rr}\vee_L \rsupp{\Foot{\qq}}$.}\,\,\right\} \label{eq:sRgTbrhtzd}
\end{align}}%
If $\rr$ is in the scope of \eqref{eq:sRgTbrhtzc} or \eqref{eq:sRgTbrhtzd}, then 
 $\rr'$ and $\rr$ are called a \emph{new edge} and a \emph{changing old edge}, respectively. 
In Figure \ref{fig-nn}, $\lsupp{\Foot{\nn_1}}=u_7$, $\rsupp\qq=v_9$, and the new edges are the red dashed ones.
It follows from \eqref{eq:XlsnmmrdhlsfTRhs} that  each edge $\rr$ of $L$ belongs to the scope of exactly one of \eqref{eq:sRgTbrhtza}--\eqref{eq:sRgTbrhtzd}.
With its new edges and the remaining old ones, $L'$ turns into a Hasse diagram of a poset $L'=(L;\leq)$, which is a subposet of $L=(L;\leq)$.
Actually, we need to verify that the diagram is a poset diagram. We need to show that no two edges of the new diagram overlap; this will be done a bit later. We also need to show that for every edge $[x,y]$ of the new diagram $L'$, there are no edges $[x,z_1]$, $[z_1,z_2]$, \dots, $[z_{k-1},y]$ of $L'$ for some $k\geq 2$. This is clear if $[x,y]$ is a new edge, as the only possible $z_1\in L$ is not in $L'$; the case when $[x,y]$ is a remaining old edge is even more obvious.  
To exclude overlapping edges and to show that the poset $L'$ is actually (the diagram of) a slim rectangular lattice, we have to work more. Since none of the original territories $\OT{\nn_1}$ and $\OT{\nn_2}$ is used, Lemmas \ref{lemma:nmtfrkstZ} and \ref{lemma:mfXzsrjsmG} imply the following.
\begin{equation}\left.
\parbox{9.3cm}{Let $i\in\set{1,2}$. Then  every edge $\rr$ in $\LEOT{\nn_i}$ is either of (normal) slope $(1,1)$ and lies on the boundary of $\LEOT{\nn_i}$ or $\rr$ is of (normal) slope $(1,-1)$. Similarly, every edge $\rr$ in $\REOT{\nn_i}$ is either of (normal) slope $(1,-1)$ and lies on the boundary of $\REOT{\nn_i}$ or $\rr$  is of (normal) slope $(1,1)$.}\,\,\right\}
\label{eq:jFszPsmjRzrF}
\end{equation}
Hence, even though $L$ can be more complicated in general than in Figure \ref{fig-nn}, the original territories indicated by appropriate fill patterns in the figure reflect the general case well. The new edges of $L'$, which originate from changing old edges of $L$,  belong to three categories, which will be discussed separately.

\begin{category}\label{categ1}
We assume that $\rr$ is a precipitous edge in the scope of  \eqref{eq:sRgTbrhtzc}. Then $\rr$ is a neon tube of a lamp $J\in\Lamp L$ such that $\Peak J=\Peak\rr$ lies on $\LFloor {\nn_2}$. In Figure \ref{fig-nn}, $J$ can be $J_1$ or $J_2$. It follows from 
\eqref{eq:jFszPsmjRzrF} that we obtain $\rr'$ from $\rr$ by moving the peak of $\rr$ to the northwest along an edge of slope $(1,-1)$. Thus, using that $\rr$ is precipitous, it follows trivially that  $\rr'$ is also precipitous; for more details, the reader can (but need not) see  \cite[(6.8)]{CzG-qdiagr}. Since no precipitous edge will occur in other categories for changing edges, 
let us summarize for later references that
\begin{equation}\left.
\parbox{9.8cm}{if a precipitous old edge $\hh$ of $L$ is a changing edge, then it changes to a precipitous new edge  $\hh'$ and $\Foot{\hh'}=\Foot\hh$.}\,\,\right\}
\label{eq:rGmnvmDlznpj}
\end{equation}

A line or an edge is of a \emph{slight slope} if it is parallel to the vector $(1,t)$ for some $t\in\RR$ such that $|t|<1$. That is, a line or edge is of a slight slope if and only if it is neither of a normal slope nor precipitous. 
We know from \cite[(6.9)]{CzG-qdiagr} (and it is easy to see) that 
\begin{equation}\left.
\parbox{8cm}{if $\ell$ is a (geometric) line through two distinct lower covers of $\Peak J$, then $\ell$ is of a slight slope.}\,\,\right\}
\label{eq:phfSzkRsvStkL} 
\end{equation}

Next, let $\UHCircR J$ stand for  the union of the 4-cells whose peaks are $\Peak J$; it is a geometric area. (The acronym, taken from \cite{CzG-qdiagr}, comes from ``upper half of the circumscribed rectangle''.) For $J\in \set{J_1,J_2}$ in Figure \ref{fig-nn}, $\UHCircR J$ in $L$ is curl-filled. 
Note that on the right of the figure, the curl-filled areas are
$\UHCircR {J_1}$ and  $\UHCircR {J_2}$ understood in $L$ but not in $L'$.
It follows from Lemmas \ref{lemma:nmtfrkstZ} and \ref{lemma:mfXzsrjsmG} (and, in a different terminology, it is explicitly stated in \cite[(6.3)]{CzG-qdiagr}) that
\begin{equation}\left.
\parbox{7cm}{$\GInt{\UHCircR J}$ contains no edge segment that is not a part of a neon tube of $J$.}
\label{eq:wCnflTgwrJkp}\,\,\right\}
\end{equation}

Practically, \eqref{eq:wCnflTgwrJkp} means that the curl-filled areas in the figure reflect generality well. Let $\hh'$ be an edge of $L'$ such that $\hh'\neq\rr'$. 
Since neither the curl-filled area $\GInt{\UHCircR J}$  nor  
the 4-cell of $\LEOT{\nn_2}$ that
is the upper left neighbor of $\CircR J$
contains an edge of $L$ not 
mentioned in \eqref{eq:wCnflTgwrJkp}, $\rr'$ neither crosses nor overlaps $\hh'$ if $\hh$ is  of a normal slope. 
Next, assume that $\hh$ is precipitous and so it is a neon tube and $\hh$ belongs to $J$, that is, to the same lamp to which $\rr$ belongs. As $\Peak{\hh'}=\Peak{\rr'}$, the edges $\hh'$ and $\rr'$ do not cross.  
It follows from \eqref{eq:phfSzkRsvStkL} (applied to the common geometric line that contains both $\hh'$ and $\rr'$) that $\hh'$ and $\rr'$ do not overlap. 
In the remaining case when $\hh$ is precipitous but not a neon tube  of $J$ and $\Peak\hh\in\LFloor{\nn_2}$, then let $K$ denote the lamp having $\hh$ as a neon tube. Then $K$ is an internal lamp and $K\neq J$. Since an internal lamp is clearly determined by its peak, $\Peak J\neq \Peak K$, and they are comparable since $\LFloor{\nn_2}$ where they belong is a chain by \eqref{eq:szmRkhs}. The role of $J$ and $K$ is interchangeable, so let $\Peak K<\Peak J$.  Then (the line determined by) $\RRoof K$ separates $J$ and $K$, and we obtain easily again that 
$\rr'$ and $\hh'$ neither cross nor overlaps. We have seen that
\begin{equation}\left.
\parbox{8.3cm}{if $\rr'$ originates from a precipitous edge $\rr$ of $L$, then $\rr'$ neither crosses nor overlaps any other edge of $L'$.}\,\,\right\}
\label{eq:kTsKrbjDmNsdBbzTs}
\end{equation}
\end{category}

\begin{category}\label{categ2} 
We assume that $\rr$ is of a normal slope and $\rr'$ is defined in \eqref{eq:sRgTbrhtzc}. Then $b:=\Peak{\rr'}\in L$ even though $\rr'$ is not an edge of $L$. It is clear either by Lemmas \ref{lemma:nmtfrkstZ} and \ref{lemma:mfXzsrjsmG} or by comparing the present situation to 
\eqref{eq:mhgPrcmgrndkKns} that $\Peak \rr \prec_L b$. Hence, $\intv d :=[\Peak\rr, b]$ is an edge. This edge lies in $\LEOT{\nn_2}$, and we obtain from \eqref{eq:jFszPsmjRzrF} that $\intv d$ is of slope $(1,-1)$.  So is $\rr$ since it is of a normal slope but does not lie on $\LFloor{\Foot{\nn_2}}$. This means that 
$\rr'$ comes to existence by merging $\rr$ and $\intv d$, which are adjacent edges lying on the same line of slope $(1,-1)$. Hence, $\rr'$ is also of slope $(1,-1)$. Therefore, since Category \ref{categ3} will be analogous to the current one by left-right symmetry and we are armed with \eqref{eq:rGmnvmDlznpj}, we can conclude even now that
\begin{equation}\left.
\parbox{9.3cm}{if $\intv g$ is a changing old edge of a normal slope, than the edge $\intv g'$ of $L'$ is of the same (normal) slope and, furthermore, $\intv g'$ is obtained by merging two collinear adjacent edges of $L$.}\,\,\right\}
\label{eq:rBwWrhnhntgRKtt}
\end{equation}
It follows from
\eqref{eq:kTsKrbjDmNsdBbzTs} and  \eqref{eq:rBwWrhnhntgRKtt} that if $\rr'$ crossed or overlapped an edge $\intv g'$ of $L'$, then $\intv g'$ would be of the other normal slope, $(1,1)$, and it would come to existence by merging $\intv g$ to a collinear other edge of $L$ at $b$. But then $\intv g$ would lie on $\RFloor{\nn_2}$ and instead of merging it to a collinear edge to obtain $\intv g'$, 
 $\intv g$ would have been omitted. Thus,
\begin{equation}\left.
\parbox{7cm}{if $\rr$ belongs to Category \ref{categ2}, then $\rr'$ neither crosses nor overlaps any other edge of $L'$.}\,\,\right\}
\label{eq:wwRtlmJncPtnsG}
\end{equation}
\end{category}

\begin{category}\label{categ3} We assume that $\rr$ is in the scope of \eqref{eq:sRgTbrhtzd}. By \eqref{eq:XlsnmmrdhlsfTRhs}, $\rr$ is of (a normal) slope $(1,1)$. Hence, the situation is basically the left-right symmetric counterpart of the one discussed in Category \ref{categ2}, whereby no details will be given.
\end{category}

Now that the three categories have been investigated, \eqref{eq:kTsKrbjDmNsdBbzTs}, \eqref{eq:wwRtlmJncPtnsG}, and the left-right symmetric counterpart of \eqref{eq:wwRtlmJncPtnsG} for Category \ref{categ3} imply that 
$L'$ is a \emph{planar} Hasse-diagram. 
We know from Kelly and Rival \cite[Corollary 2.4]{KR75} that planar posets with 0 and 1 are lattices. Hence, $L'$ is a planar lattice.
By construction, the number of upper covers of an element $x\in L'$ is the same in $L'$ as in $L$. Furthermore, an element of $L'$ belongs to the boundary of $L'$ if and only if it belongs to the boundary of $L$. 
Therefore, \eqref{eq:mnHhwlCh} and
the construction of $L'$ yield in a straightforward but a bit tedious way that $L'$ is a slim rectangular lattice.  

Since $x\in L'$ has the same number of covers in $L'$ as in $L$, we obtain that $\Mir{L'}=L'\cap\Mir L$. Moreover, we already have  \eqref{eq:rGmnvmDlznpj} and \eqref{eq:rBwWrhnhntgRKtt}, and it is clear that  an edge  $\rr'$ of $L'$ lies on  $\Bnd{L'}$ if and only if it lies on $\Bnd L$. Clearly, $\lcorner L,\rcorner L\in L'$. Therefore, 
taking the just mentioned facts of the present paragraph and  Convention \ref{conv:tbdD} (for $L$) into account, we conclude that $L'$ is (given by) a \tbdia-diagram.  

Since $\OT{\nn_2}$ is not used, it follows from \eqref{eq:szmRkhs} and   Lemma \ref{lemma:thtznjtVnPrm} that 
\begin{equation}
\text{if $\hh$ is a neon tube of $L$ and $\hh\neq\nn_2$, then $\Foot\hh\notin F(\nn_2)=\Floor{\nn_2}$.}
\label{eq:kcSgZnKgncsZs}
\end{equation}
It follows from \eqref{eq:rGmnvmDlznpj}, \eqref{eq:rBwWrhnhntgRKtt}, and the construction of $L'$ that
\begin{equation}\left. 
\parbox{9.5cm}{the neon tubes of $L'$ are exactly the $\rr'$ where $\rr$ is a neon tube of $L$ and $\rr\neq \nn_2$.  Furthermore, for neon tubes $\rr$ and $\hh$ of $L$ such that $\rr\neq\nn_2\neq \hh$, $\Peak{\rr'}=\Peak{\hh'}$ if and only if $\Peak\rr=\Peak\hh$ and $\Foot{\rr'}=\Foot\rr$.}\,\,\right\}
\label{eq:BHtgjLrSzglT}
\end{equation}
Hence, for a lamp $K\in\Lamp L\setminus\set I$, $\set{\rr':\rr\text{ is a neon tube of }K}$ is exactly the collection of neon tubes of a lamp $K'$ of $L'$. Furthermore, $\{\hh: \hh$  is a neon tube of $I$  and $\hh\neq\nn_2\}$ is the set of neon tubes of an internal lamp $I'$ of $L'$ --- this is the definition of $I'$. 
Note that Lemma \ref{lemma:mThFgZzbBRt} and \eqref{eq:BHtgjLrSzglT} give that  $\Foot{K'}=\Foot K$ for $K\in \Lamp L\setminus\set I$.  Now  \eqref{eq:BHtgjLrSzglT} 
and the facts mentioned thereafter allow us to conclude that the function $\phi\colon\Lamp L\to\Lamp{L'}$ defined by
\begin{equation}
K\mapsto 
\begin{cases}
K'&\text{if $K'\in\Lamp{L'}$ such that $\Foot {K'}=\Foot K$,}\cr
I'&\text{if $K=I$}
\end{cases}
\label{eq:kTngdfhkBlJt}
\end{equation}
is bijective. (Remark that if $\nn_2$ is not the rightmost neon tube of $I$, then $I$ belongs to the scope of both lines of \eqref{eq:kTngdfhkBlJt}.)
Note the rule, which follows from \eqref{eq:BHtgjLrSzglT}: for any $K\in\Lamp L$,  we have that $\Peak{\phi(K)}=\Peak K$. 

We know from Lemma \ref{lemma:vltfLm} that, in order to see that $\phi$ is an order isomorphism, it suffices to show that, for $J,K\in\Lamp K$, 
\begin{equation}
(J,K)\in\rhfoot \iff (J',K')\in\rhfoot.
\label{eq:glHrvlBkmlnDlj}
\end{equation} 

Assume that $(J,K)\in\rhfoot$ and $J\neq I$. Since $\Peak {K'}$ is to the northwest (that is, to the $(-1,1)$ direction) of $\Peak K$ or $\Peak{K'}=\Peak K$, we have that $\Enl K\subseteq \Enl {K'}$. 
Hence, $\Foot {J'}=\Foot J \in\Enl K\subseteq \Enl{K'} $ gives the required 
 $(J',K')\in\rhfoot$. If $(I,K)\in\rhfoot$, then $\CircR {I'}=\CircR I\subseteq \Enl K\subseteq \Enl{K'}$ by Lemma \ref{lemma:vltfLm}, whereby $(I',K')\in\rhcircr=\rhfoot$, as required. This proves the ``$\then$'' part of \eqref{eq:glHrvlBkmlnDlj}.

Next, assume that $(J',K')\in\rhfoot$ and $I\notin\set{J,K}$. We know that $\Foot{K'}=\Foot K$ and $\Foot{J'}=\Foot J$. 
If $\Peak{K'}=\Peak K$, then 
$\Foot{J}=\Foot{J'}\in\Enl{K'}=\Enl{K}$ gives the required $(J,K)\in\rhfoot$. 
So assume that $\Peak{K'}\neq\Peak K$. 
By construction, $\Enl {K'}\subseteq \Enl K\cup \LEOT{\nn_2}$; see Figure \ref{fig-nn}. Hence, $\Foot{J}=\Foot{J'}\in\Enl{K'}$ gives that $\Foot J\in \Enl K$ or $\Foot J\in\LEOT{\nn_2}$. 
If the second alternative, $\Foot J\in\LEOT{\nn_2}$, holds, then $\Foot J\subseteq \EOT{\nn_2}$, which contradicts Lemma \ref{lemma:thtznjtVnPrm} as $\OT{\nn_2}$ is not used. 
Hence, $\Foot J\in \Enl K$, which gives that $(J,K)\in\rhfoot$, as required. 

We are left with the case when one of $J$ and $K$ is $I$.

Assume that $(J',I')\in\rhfoot$. Then 
$\Foot J=\Foot {J'}\in\Enl{I'}\subseteq\Enl I$ gives the required $(J,I)\in\rhfoot$. (Note that $\Enl {I'}\subset \Enl I$ if $\nn_2$ is the rightmost neon tube of $I$, and $\Enl {I'}= \Enl I$ otherwise.)

Finally,  assume that $(I',K')\in\rhfoot$. 
Then  $(I',K')\in\rhcircr$ by  Lemma \ref{lemma:vltfLm}. This fact and $\CircR I=\CircR{I'}$ give that 
\begin{equation*}\Peak I=\Peak{\CircR{I}}=\Peak{\CircR{I'}}
\in\CircR{I'} \subseteq \Enl {K'}.
\end{equation*}
Hence, $(\Foot{K'},\Peak{K'})=(\Foot{K},\Peak{K})$, and so $\Enl{K'}=\Enl K$. These facts lead to $\CircR I=\CircR {I'}\subseteq \Enl {K'}=\Enl K$.  Thus, $(I,K)\in\rhcircr=\rhfoot$, as required. The proof of Lemma \ref{lemma:twnTbs} is complete.
\end{proof}

\section{An estimate}
The \emph{length} of a lattice $K$ is denoted by $\len K$.  Our goal is to prove that

\begin{theorem}\label{thm:est} 
Let $D$ be a ConSPS-representable distributive lattice with 
$n:=|\Jir D|$  join-irreducible elements. If $n\in\set{0,1}$, then $D$ is the $(n+1)$-element chain and $K\cong D$.
If $n=2$, then $D$ is the four-element boolean lattice and either $K\cong D$ or $K$ is the three-element chain.
If   $n\geq 3$,  then the following two assertions hold.

\textup{(A)} There is a slim rectangular lattice $L$ such that  $\Con L\cong D$ and 
\begin{equation}
\len L\leq 2n^2-10n + 15,\quad\text {and so }\quad \len L < 2n^2. 
\label{eq:MzlrGnvhPtBl}
\end{equation}

\textup{(B)} For any slim semimodular lattice $L'$, if \ $\Con{L'}\cong D$, then $\len{L'}\geq n$.
\end{theorem}

\begin{proof} The case $n\leq 2$ is trivial. In the rest of the proof, let $n\geq 3$.
Let $L$ be a slim rectangular lattice.
A trivial induction by Lemmas \ref{lemma:nmtfrkstZ} and \ref{lemma:mfXzsrjsmG} shows that
\begin{equation}
\len L=\ANTube L=|\Mir L|.
\label{eq:kZmbTrMknHgl}
\end{equation}
Now if $\Con {L}\cong D$, then $\Lamp L\cong \Jir D$ by Lemma \ref{lemma:vltfLm},  so \eqref{eq:kZmbTrMknHgl}
gives that $\len L=\sum_{I\in\Lamp L}\NTube I\geq \sum_{I\in\Lamp L} 1=|\Lamp L|=n$.
Hence, Part (B) holds for the particular case of rectangular SPS lattices.

We know from  Gr\"atzer and Knapp \cite[Theorem 7]{GKnapp-III} \emph{and its proof} that
\begin{equation}\left.
\parbox{8cm}{each slim semimodular lattice $L'$ with at least three elements is a sublattice of  a  slim rectangular lattice $L$ such that $\Con L\cong\Con{L'}$ and $\len L=\len {L'}$.}\,\,\right\}
\label{eq:kTkSMrtkr}
\end{equation}
This statement also follows from Cz\'edli and Schmidt \cite[Lemma 21]{czgscht-visual} (applied in the reverse directions) and 
 Cz\'edli\cite[(Corner) Lemma 5.4]{CzG:repHom}.  Therefore, Part (B) follows from its particular case mentioned above.

Next, we turn our attention to part (A). We can assume that $\Jir D$ is not an antichain since otherwise with any grid $G$ of length $n$ and $L:=G$, we have that $\Con G\cong D$ and $\len G=n\leq 2n^2$.
Take a slim rectangular lattice $L$ of minimal length such that $\Con L\cong D$. We know from Lemma \ref{lemma:vltfLm} that $\Lamp L\cong \Jir{D}$, and so $|\Lamp L|=n$. Let $J\in \Lamp L$ be an internal lamp. Let $t^+_J$ 
denote the number of neon tubes of $J$ whose original territories are used. Similarly, $t^-_J$ stands for the number of neon tubes of $J$ whose original territories are not used; note that $t_J^+ + t_j^-=\NTube J$. Listing the neon tubes from left to right, let us write a letter $u$ for a used neon tube and a zero for an unused neon tube.
Then we obtain a sequence $\vec s$ of length $\NTube J$ consisting of $t_J^+$ $u$'s and $t_J^-$ zeros. Subsequences $0\,u\,0$ and $0\,0$ are forbidden by \eqref{eq:kZmbTrMknHgl} and Lemmas \ref{lemma:middllunt} and \ref{lemma:twnTbs} since $\len L$ is minimal. For another look at $\vec s$, take the sequence $\vec w:= \phantom u\star u \star u \star u\dots \star u\star u \star u\star\phantom u$
of $t_J^+$ $u$'s and $t_J^+ +1$ stars that alternate. 
We can  obtain $\vec s$ from $\vec w$ by removing some stars and replacing the remaining stars by zeros. Observe that only one zero can replace a star since $0\,0$ is a forbidden subsequence. Furthermore, for any two consecutive stars (which occur in a subsequence $\star\, u\, \star$), at most one of the two stars can change to $0$ and so the other one should be removed since $0\,u\,0$ cannot be a subsequence.  Hence, at most every second star can turn to 0 and the rest of the stars are removed. Therefore, the number $t_J^-$ of zeros is at most\footnote{Provided that $t_J^+ > 0$; this correction will be taken into account about seven lines after \eqref{eq:sWrhkkStDBnr}.} $\lceil (t_J^++1)/2\rceil$,
the upper integer part of $(t_J^++1)/2$. Since $\lceil (t_J^++1)/2\rceil\leq t_J^+$, we obtain that, for any $J\in\Lamp L$, 
\begin{equation}
\NTube J = t_J^+ + t_j^-  \leq   2\cdot t_J^+.
\label{eq:nsFlRnknFsTdrrsslH}
\end{equation}

Let $m$ denote the number of boundary lamps, that is, the number of maximal elements of $\Lamp L$ (or, equivalently, those of $\Jir D$). 
Each of $\LBnd L$ and $\RBnd L$ contains at least one boundary lamp, whence $m\geq 2$. 
Since $\Lamp L\cong \Jir D$ is not an antichain, $m<n$. So $k:=n-m$, the number of internal lamps of $L$, is at least 1.
If $\pp$ is a neon tube of an internal lamp $J$ and $I$ uses the original territory of $J$, then $I<J$ and, in particular, $I$ is also an internal lamp. Furthermore, if $\pp_1$,\dots, $\pp_{t_J^+}$ denote the neon tubes of $J$ whose original territories are used, then the $\GInt{\LEOT{\pp_1}}$, \dots, $\GInt{\LEOT{\pp_{t_J^+}}}$ are pairwise disjoint, and so are  $\GInt{\REOT{\pp_1}}$, \dots, $\GInt{\REOT{\pp_{t_J^+}}}$.  Therefore, using Lemma \ref{lemma:thtznjtVnPrm}(b), it follows that the lamp $I$ can use the original territories of at most two of the neon tubes of $J$.
The number of lamps $I$ that use the original territory of a neon tube of $J$ is at most
$|\ideal J\setminus\set J|$, whereby $J$ has at most $2\cdot |\ideal J\setminus\set J|$ 
neon tubes\footnote{For minimal lamps, this will be corrected soon.} whose original territories are used. By \eqref{eq:nsFlRnknFsTdrrsslH}, it has at most twice as many neon tubes all together. 
Hence, the total number of neon tubes of the internal lamps is at most\footnote{To  be improved  soon by taking the minimal internal lamps of $L$ into account.}
\begin{equation}
\sum_{\text{internal }J\in\Lamp L}2\cdot 2\cdot |\ideal J\setminus\set J| = 4\cdot \sum_{\text{internal }J\in\Lamp L} |\ideal J\setminus\set J|.
\label{eq:sWrhkkStDBnr}
\end{equation} 
Observe that $|\ideal J\setminus\set J|$ is the number of pairs $(I,I')$ of internal lamps subject to $I<I'$ and $I'=J$. Therefore, the second sum in \eqref{eq:sWrhkkStDBnr} is the number of pairs $(I,J)$ of internal lamps such that $I<J$. 
This sum reaches its maximum  when the internal lamps form a chain. Then there are ${k \choose 2}=k(k-1)/2$ such pairs, and so the maximum that \eqref{eq:sWrhkkStDBnr} can take is $2k(k-1)$; it \emph{might seem} to be an upper bound on the number $\INTube L$ of the neon tubes of the internal lamps of $L$.  

There are two imperfections with the argument above.  First, any two minimal internal lamps are incomparable. Hence, letting $s$ denote the number of minimal internal lamps, 
${k\choose 2} =k(k-1)/2$ has to be reduced by ${s\choose 2} =s(s-1)/2$. Second, instead of $2\cdot |\ideal J\setminus\set J|=0$, a minimal lamp $J$ has  exactly one neon tube (trivially or by Lemma \ref{lemma:middllunt}), whereby we $s\cdot 1=s$ has to be added. 
So we obtain that
\begin{align}
\INTube L&\leq 4\cdot\bigl(k (k-1)/2-s (s-1)/2\bigr) +s \cr
&= 2k^2-2k+3s-2s^2\leq'  2k^2-2k+1, 
\label{eq:wzMrkWf}
\end{align}
where ``$\leq'$'' holds since $3s-2s^2$ is negative for $s\geq 2$ and so we substituted 1 for $s$. 

Next, taking the $m$ boundary lamps,  $k=n-m$, and \eqref{eq:wzMrkWf} into account, 
\begin{align}
\ANTube L&=m+\INTube L  \cr
&\leq m + 2(n-m)^2-2(n-m)+1 \cr
&=2n^2-2n+1 +  2\cdot \underbrace{\Bigl(m^2- (2n-3/2)m\Bigr)}.
\label{eq:mNbnShjgtjmLmSg}
\end{align}
Let $f(m)=m^2- (2n-3/2)m$ denote the under-braced term. 
By the elementary theory of quadratic univariate real functions, $f(m)$  decreases in the closed interval $[0, n-3/4]$. This fact and  $2\leq m\leq n-1$ imply that  the largest value of $f(m)$ is $f(2)= 7-4n$. Substituting this value into \eqref{eq:mNbnShjgtjmLmSg}, we obtain that 
\begin{equation}
\ANTube L\leq 2n^2-10n + 15 < 2n^2.
\label{eq:psTsrNkzmRtcZl}
\end{equation}
Finally, \eqref{eq:kZmbTrMknHgl} and \eqref{eq:psTsrNkzmRtcZl} complete the proof of Theorem \ref{thm:est}.
\end{proof}

\begin{remark}\label{rem:rTsnkLbztsnlGpk}
The inequality \eqref{eq:MzlrGnvhPtBl} is not sharp. Indeed, no matter which $4$-element poset $\Jir D$ is, there is a slim rectangular lattice  $L$ such that $|\Jir{\Con L}|\cong D$ and $\len L\leq 5$ while
$2n^2-10n+15$ for $n:=4$ is 7. Note that ``$\leq 5$'' is sharp for $n=4$; to see this, let $\Jir D$ be the 4-element poset with the  ``Y-shaped diagram''.
\end{remark}

\begin{corollary}\label{cor:hcGprltKsJtkb}
For $L$ in Part \textup{(A)} of Theorem \ref{thm:est},  $|L|\leq (2n^2-10n+15)^2<4n^4.$
\end{corollary}

\begin{proof} By \eqref{eq:kTkSMrtkr} and Theorem \ref{thm:est}, it suffices to show that if $L$ is a slim rectangular lattice  of length $k$, then $|L|\leq k^2$. 
By  \eqref{eq:sZbjmRgRpbW}, there are chains $C,U\subseteq \Jir L$ such that $\Jir L=C\cup U$. Since $0\notin C$ and, by rectangularity, $1\notin C$,   $|C|\leq k-1$. Similarly, $|U|\leq k-1$. Since any element of $L\setminus\set 0$ is of the form $c\vee u$ with $c\in C$ and $u\in U$, $L$ has at most $1+|C|\cdot|U|=1+(k-1)^2\leq k^2$ elements, completing the proof.
\end{proof}

\section{Odds and ends}
Let $P$ be a poset, and let $j\in P$. We define a new poset $P'$ as follows. The base set of $P'$ is $(P\setminus\set j)\cup \set{j',j''}$ where $P\cap\set{j',j''}=\emptyset$. The ordering in $P'$ is defined as follows: 
for $a,b\in P'\setminus\set{j',j''}=P\setminus\set j$,  $a\leq_{P'} b\iff a\leq_P b$, $a\leq_{P'} j'\iff a\leq_{P'} j''\iff  a \leq_{P} j$, \ $j'\leq_{P'}b\iff j''\leq_{P'}b\iff j\leq_P b$, and $j''\prec_{P'} j'$.
We say that $P'$ is obtained from $P$ by \emph{doubling the element} $j$ of $P$. For an example, see $P$ and $P'$ in the middle of Figure \ref{fig-ddb2}.

\begin{proposition}\label{prop-dbl} 
Let   $P'$ be a poset obtained from a JConSPS-representable poset $P$ by doubling a non-maximal element $j\in P$. Then  $P'$
is also JConSPS-representable.
Furthermore, if $L$ is a slim rectangular lattice such that 
$P\cong\Jir{\Con L}$, then there is a slim rectangular lattice $L'$ such that 
$P'\cong\Jir{\Con{L'}}$ and $\len{L'}=\len L +2$.
\end{proposition}

\begin{figure}[ht] \centerline{ \includegraphics[scale=0.98]{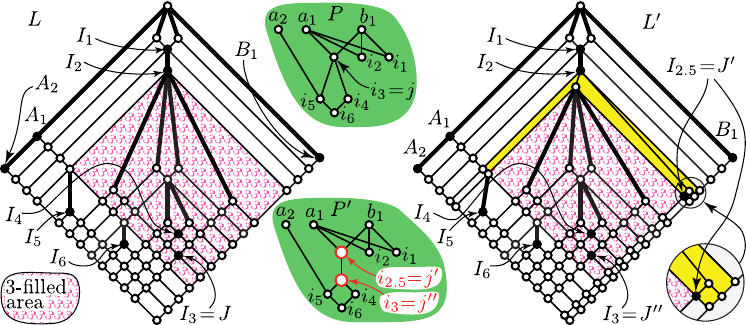}} \caption{The construction for Proposition \ref{prop-dbl} 
with a ``magnifying glass'' at the bottom right}\label{fig-ddb2}
\end{figure}

Cz\'edli \cite[Corollary 3.5]{CzGlamps} shows that 
if we double a \emph{maximal} element of a JConSPS-representable poset $P$, then the new poset $P'$ is never JConSPS-representable.

\begin{proof}[Proof of  Proposition \ref{prop-dbl}] By Gr\"atzer and Knapp's result, see \eqref{eq:kTkSMrtkr}, it suffices to deal with the second half of the statement. Assume that $L$ is a rectangular lattice. 
For $m\in\Nplu$, the $m$-th neon tube of a lamp $I$ is understood as the $m$-th neon tube of $I$ from the left; see Convention \ref{conv:tbdD}.
We also count on the fixed multifork sequence of $L$, see Lemmas \ref{lemma:nmtfrkstZ} and \ref{lemma:mfXzsrjsmG}. We know from Lemma \ref{lemma:vltfLm} that there is an order isomorphism $P\to \Lamp L$; we denote its action by capitalization, that is, $x\mapsto X$. The notation used in Lemma \ref{lemma:mfXzsrjsmG} is in effect. Since $j$ is not a maximal element of $P$,   $J$ is an internal lamp; let, say, $J=I_t$. In Figures \ref{fig-ddb2} and\footnote{Apart from scaling, the two figures are the same. Figure  \ref{fig-ddb2} illustrates the idea of the construction better while Figure \ref{fig-ddb1} is more readable.} \ref{fig-ddb1}, $t=3$. 
Note that $P\cap P'=P\setminus\set{j}=P\setminus\set{j',j''}$ is a subposet both in $P$ and in $P'$. 
For any  $x\in P\cap P'$,  the lamp corresponding to $x$
will be denoted by $X$ both in $L$ and in $L'$; this should not cause confusion since it will be clear from the context whether $X\in \Lamp L$ or $X\in \Lamp{L'}$. 
The pair $(\Foot X,\Peak X)$ is the same in $L'$ as in $L$. So, implicitly, the proof mostly  considers lamps as pairs.

We define $L'$ in the following way.
Let $\epsilon\in\RR$, $\epsilon>0$, be the smallest one out of the geometric lengths of the edges of (the fixed \tbdia-diagram of) $L$.  
With reference to the multifork sequence of $L$, let 
$L_0':=L_0$, $L_1':=L_1$,\dots, $L'_{t-1}:=L_{t-1}$; these equations also mean the exact coincidence of the corresponding \tbdia-diagrams in the plane. As for the forthcoming notation, we will continue the sequence by $L'_{t-0.5}$, $L'_t$, $L'_{t+1}$, \dots, $L'_k=:L'$.
In $L'_{t-1}$ (which is the  same as $L_{t-1}$), let  $H'_{t-0.5}$ be the same 4-cell (even geometrically the same) as $H_{t}$ in $L_{t-1}$. 
 
Later, $H_t$ turns into $\CircR{I_t}$ in $L$; in the figure, $\CircR{I_t}=\CircR{I_3}$ is the ``3-filled'' area in $L$. In $L'$, only the ``major part'' of $\CircR{I'_{t-0.5}}=\CircR{I'_{2.5}}$ is 3-filled; the rest of $\CircR{I'_{t-0.5}}=\CircR{I'_{2.5}}$ is yellow-filled.
At $H_t$ in $L_{t-1}$, we perform a $\NTube{I_t}$-fold multifork extension, which produces $J=I_t$. 
(In the figure, where $I_t=I_3=J$, $\NTube{I_t}=4$.) However, in $L'_{t-1}$, we add a 2-fold multifork at $H'_{t-0.5}$ to obtain a new lattice $L'_{t-0.5}$. Geometrically (in the \tbdia-diagram), this  new multifork extension and the lamp $J'=I_{t-0.5}$ it produces look unusual compared to other figures.   Namely, we require that the 4-cell  $H'_{t}$ whose peak is the foot of the leftmost neon tube of $J'$ should be almost as large as $H'_{t-0.5}$. That is, the width $\eta$ of the ``legs'' of the $\Lambda$-shaped difference $H'_{t-0.5}\setminus H'_{t}$, which is yellow-filled in the figure, should be very small. (We may think of $\eta=\epsilon/1000$.) On the right of  the Figure, $H'_t=H'_3$ in $L'$ is 3-filled. 

Next, we perform a $\NTube{I_t}$-fold multifork extension at $H'_t$ to obtain $L'_t$ from $L_{t-0.5}$ and to produce the lamp $J''=I_t$ of $L'_t$ (and of $L'$). The feet of the neon tubes of $J''=I_t$ in $L'_t$ (and in $L'$) should be the same geometric points as the feet of the neon tubes of $J=I_t$ in $L_t$ (and in $L$). So the geometric shape of $J$  and that of $J''$ are almost the same (and they tend to be the same as $\eta$ tends to 0).

From $L'_t$, we continue the multifork sequence for $L'$ in the same way as we continue the sequence from $L_t$ to reach $L$. Even in geometric sense, we do almost the same, that is, with very little differences that would diminish if we formed the limit at $\eta\to 0$.
To be more specific, let us agree that we use the alternative notation $I_{-1}=A_1$, $I_{-2}=B_1$, $I_{-3}=A_2$, $I_{-4}=B_2$, \dots{}, $I_{-2k+1}=A_k$, $I_{-2k}=B_k$, \dots{} for the boundary lamps. (The purpose of this notation is that now each lamp is of the form $I_m$ for some $m\in\mathbb R$.)
For $s=t, t+1,\dots,k-1$, 
we select $H'_{s+1}$ as follows. In $L_{s}$, the trajectory through the top left edge of the 4-cell $H_{s+1}$ contains exactly one neon tube, $\pp$. Since the top left edge of $H_{s+1}$ is of slope $(1,1)$, it is in the descending part of the trajectory. 
The neon tube $\pp$ belongs to exactly one lamp, which is older than or as old as $I_s$; let $I_u$ denote this lamp.  Note that we never use the trajectory through the leftmost neon tube of $I_{t-0.5}$ (in the figure, the ``narrow'' trajectory through the yellow-filled area), whereby $u\neq t-0.5$ and so $u$ is an integer and $I_u$ will also make sense in $L'$, not only in $L$. 

Among the neon tubes of $I_u$, let $\pp$ be the $\alpha$-th neon tube (from the left).  In $L'_t$, let $\pp'$ be the $\alpha$-th neon tube of $I_u$.    By left-right symmetry, the top right edge of $H_{s+1}$ defines a neon tube $\qq$ of a lamp $I_v$ in $L_s$ and its counterpart $\qq'$ in $L'_s$. The top right edge of $H_{s+1}$ is in the ascending part of the trajectory in question.
Now we can simply select $H'_{s+1}$ as the unique 4-cell of $L'_s$ where the descending part of  the  trajectory through $\pp'$ and the ascending part of the trajectory through $\qq'$ cross each other\footnote{The possible doubts whether they cross will be dissolved later.}. Once $H'_{s+1}$ has been selected, we perform a $\NTube{I_{s+1}}$-fold multifork extension at this 4-cell of $L'_s$  to obtain $L'_{s+1}$ and its lamp $I_{s+1}$. 
This multifork extension should  almost be the same geometrically as in the passage from $L_s$ to $L_{s+1}$; in particular, the feet of the new neon tubes have to be geometrically the same in $L'_{s+1}$ as in $L_{s+1}$. For later reference, note that 
\begin{equation}\left.
\parbox{8.5cm}
{the left upper edge of $\CircR{I_{s+1}}=H_{s+1}$ belongs to the trajectory through  a neon tube of $I_u$ both in $L$ an $L'$, and similarly for the right upper edge and $I_v$.}\,\,\right\}
\label{eq:nmCsmDrglfvZbtD}
\end{equation}
Finally, we obtain $L'=L'_k$.

Next, in order to  recall Cz\'edli \cite[Lemma 7.5]{CzG-qdiagr}, we need some notation. Let $U$ be an internal lamp of a   slim rectangular lattice $K$. Then the top edge of the trajectory containing the upper left edge of $\CircR U$ is a neon tube of a lamp; we denote this lamp by $\Nwl U$. Left-right symmetrically, $\Nel U$ stands for the unique lamp that has a neon tube whose trajectory contains the upper right edge of $\CircR U$.  For a poset $Q$, let $\Min Q$ stand for the set of minimal elements of $Q$. Now \cite[Lemma 7.5]{CzG-qdiagr} asserts that  if $K$ is a slim rectangular lattice and $U,V\in\Lamp K$, then
\begin{equation}\left.
\parbox{7.9cm}{$U\prec V$ in $\Lamp K$ if and only if $U$ is an internal lamp and $V\in \Min{\set{\Nwl U,\Nel U}}$.}\,\,\right\}
\label{eq:tjFlTjMZvsBl}
\end{equation}
Comparing \eqref{eq:nmCsmDrglfvZbtD} and  \eqref{eq:tjFlTjMZvsBl} and taking into account that only internal lamps, which all occur in \eqref{eq:nmCsmDrglfvZbtD}, can be covered by another lamp, 
the construction implies 
that $\Lamp L \setminus \set J$ is order isomorphic to $\Lamp{L'}\setminus\set{J',J''}$. We obtain from Lemma \ref{lemma:vltfLm} that $J'<J''$ in $\Lamp{L'}$, $\Lamp L\cong \Lamp{L'}\setminus\set{J'}$, and $\Lamp L\cong \Lamp{L'}\setminus\set{J''}$. Thus, using that $P\cong\Lamp{L}$, we conclude that $P'\cong\Lamp{L'}$, as required. Furthermore, the  construction and \eqref{eq:kZmbTrMknHgl} yield that $\len{L'}=\len L+2$.

However, the proof is not complete yet. Indeed, we need to show that the trajectories mentioned earlier do cross in $L'_{s}$. To be more precise, we need to show that if the geometric areas
$\REOT\pp$ and $\LEOT\qq$ cross in $L_s$, than so do $\REOT{\pp'}$ and $\LEOT{\qq'}$ in $L'_s$. Of course, $\REOT{\pp'}$ and $\LEOT{\qq'}$ are perpendicular if we disregard their thickness but, in principle, they could avoid each other like the right leg of the upper $\pmb{\wedge}$ and the left leg of the lower $\pmb\wedge$
do in 
\begin{equation}
\text{\raisebox{8pt}{$\pmb{\bigwedge}$}\kern-11pt $\pmb{\bigwedge}$}\,\,\,\, .
\label{eq:flRjJnLsFplWv}
\end{equation}
Fortunately, it is clear by continuity that whenever $\eta$ is small enough (compared to $\epsilon$), then $\REOT{\pp'}$ and $\LEOT{\qq'}$ are close enough to $\REOT{\pp}$ and $\LEOT{\qq}$, respectively. Thus, since $\REOT{\pp}$ and $\LEOT{\qq}$ cross each other at a rectangle with sides at least $\epsilon$, $\REOT{\pp'}\cap\LEOT{\qq'}$ is a rectangle of a positive area. Furthermore, in $L_s$, 
$\REOT{\pp}\cap\LEOT{\qq}$ is a 4-cell. Since, except when $J''=I_t$ was created, $\OT{J'}=\OT{I_{t-0.5}}$ is never used, we conclude that $\REOT{\pp'}\cap\LEOT{\qq'}$ is also a 4-cell. This shows that the definition of $L'_{s+1}$ and that of $L'$ make sense, completing the proof of Proposition \ref{prop-dbl}.
\end{proof}

\begin{figure}[ht] \centerline{ \includegraphics[scale=0.98]{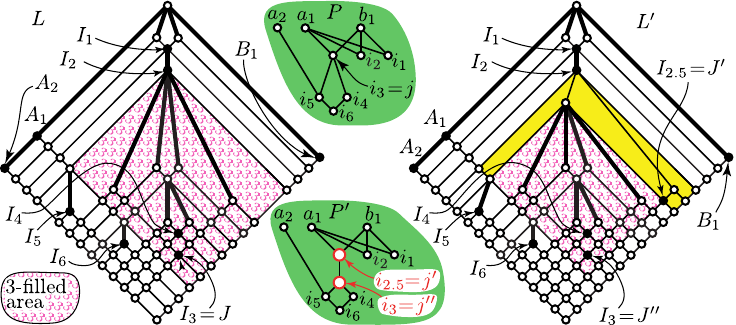}} \caption{The construction for Proposition \ref{prop-dbl}, rescaled}\label{fig-ddb1}
\end{figure}

\begin{remark}\label{rem:sldnTdfjblnS} 
In most of the cases, the estimate given in \eqref{eq:MzlrGnvhPtBl} of Theorem \ref{thm:est} is far from being optimal. For example, if $\Jir{\Con{L'}}\cong \Jir D\cong P'$ and $P'$ is obtained from a smaller poset $P$ by doubling a
non-maximal element $j\in P$, then, with the notation of  Proposition \ref{prop-dbl}, the lamp $J'$ corresponding to $j'\in P'$ has only two neon tubes and contributes to $\len{L'}$ by 2 regardless the size of $\pideal{\Lamp{L'}}{J'}$. 

To present another example, let $4\leq n\in \Nplu$ and let $P_n$ be the $n$-element poset consisting of two maximal elements, $a$ and $b$, $n-3$ minimal elements, $c_1$, \dots, $c_{n-3}$, and an element $u$ such that $u\prec a$, $u\prec b$, and $c_i\prec u$ for all $i\in\set{1,\dots,n-3}$.
 Then there is a slim rectangular lattice $L$ such that $\Jir{\Con L}\cong P_n$ and 
$\len L=n+1$, which is much smaller than what the estimate  \eqref{eq:MzlrGnvhPtBl} gives. 

In our third example, $3\leq n\in\Nplu$ and $Q_n$ is the poset with two maximal elements and $n-2$ minimal elements such that every minimal element is covered by both maximal elements. Then there is a slim rectangular lattice $L$ such that $\Jir{\Con L}\cong Q_n$ and $\len L=n$. This example shows that  the lower estimate given in Theorem \ref{thm:est}(B) cannot be improved.
\end{remark}
As Remarks \ref{rem:rTsnkLbztsnlGpk} and \ref{rem:sldnTdfjblnS} allow us to guess,  there are many factors that can reduce the number $\len L=|\ANTube L|$ and improve the estimate \eqref{eq:MzlrGnvhPtBl}. However, it seems to be difficult to take more factors into account without making Theorem \ref{thm:est} and the corresponding proof too complicated. 
Corollary \ref{cor:hcGprltKsJtkb} is not sharp either. Indeed, in addition to that 
this corollary is built on the non-sharp Theorem \ref{thm:est}, there is another reason for  this. Namely, if $\Jir D\cong \Jir{\Con L}$ has few non-maximal elements (in particular, if $\Jir D$ is an antichain and so $D$ is Boolean), then $|L|$ has few internal lamps and $|L|$ is close to $\len L^2$ but then $\len L$ is much smaller than what \eqref{eq:MzlrGnvhPtBl} gives.
On the other hand, if $\Jir D$ has many non-maximal elements, then $L$ has many internal lamps and  $|L|$ is considerably smaller than $\len L^2$.

\begin{remark}\label{rem:rjmdnMgzKrrGv} 
In order to decide whether a given $n$-element poset $P$  is JConSPS-representable, it is not economic and usually not even feasible to list all slim rectangular lattices of lengths at most $2n^2-10n+15$; see  \eqref{eq:gNsztdLnkGsllSskl} and  \eqref{eq:MzlrGnvhPtBl}, or those
 of size at most $(2n^2-10n+15)^2$; see Corollary \ref{cor:hcGprltKsJtkb}. It is much faster to rely on the known properties and constructions.
To \emph{exclude} the JConSPS-representability of $P$ in many cases, we can
check the \emph{known properties} of JConSPS-representable posets, see \eqref{eq:kTkSMrtkr},  Cz\'edli \cite{CzGlamps}, \cite{CzG-DCElamps}, and Cz\'edli and Gr\"atzer \cite{CzGGG3p3c} (where two earlier properties  from Gr\"atzer \cite{gG-cong-fork-ext} and \cite{gG-VIII-conSPS} are also recalled). 
To \emph{conclude} the JConSPS-representability of $P$ and \emph{to obtain} a slim rectangular lattice $L$ such that $P\cong\Jir{\Con L}$, we can often use the \emph{known constructions};  see Proposition \ref{prop-dbl},  Cz\'edli \cite[Theorems 3.14 and 3.16]{CzG-qdiagr}, and Cz\'edli and Gr\"atzer \cite[Theorem 1.2]{CzGGG3p3c}.  
If the known properties and constructions do not help, then,  compared to what \eqref{thm:est}  gives,
the ideas in their proofs 
\emph{radically reduce} the number of cases to be inspected for the given $P$.
\end{remark}

If $|P|$ is a small poset, then Remark \ref{rem:rjmdnMgzKrrGv} offers a way to decide, in few hours without computers, whether $P$ is JConSPS-representable. 
(We feel but have not checked that every at most 6-element poset is small in this aspect.)
Note that by Cz\'edli \cite[Corollary 3.11]{CzG-qdiagr}, 
 each finite poset $P$ that is not JConSPS-representable  gives a property (but not always a new property) of JConSPS-representable posets.

\end{document}